\documentclass{amsart}
\usepackage{amsthm}
\usepackage{amssymb, latexsym}
\usepackage{enumerate}
\usepackage{mathtools,color}
\usepackage{dsfont}
\usepackage{bbm}

\newcommand\card{\operatorname{card}}

\makeatletter
\newcommand\footnoteref[1]{\protected@xdef\@thefnmark{\ref{#1}}\@footnotemark}
\makeatother

\newtheorem{prop}{Proposition}
\newtheorem{thm}[prop]{Theorem}
\newtheorem{rem}{Remark}
\newtheorem{lem}[prop]{Lemma}
\newtheorem{res}{Result}

\begin{document}

\title[Focusing Barely Supercritical Schr\"odinger Equations] {Scattering above energy norm of a
focusing size-dependent log energy-supercritical Schr\"{o}dinger equation with radial data below ground state}
\author{Tristan Roy}
\address{American University of Beirut, Department of Mathematics}
\email{tr14@aub.edu.lb}

\begin{abstract}
Given $n \in \{ 3,4,5 \}$ and $k > 1$ (resp. $ \frac{4}{3} > k > 1$) if $n \in \{ 3,4 \}$ (resp. $n=5$), we prove scattering of the radial $\tilde{H}^{k} := \dot{H}^{k} (\mathbb{R}^{n}) \cap \dot{H}^{1} (\mathbb{R}^{n})-$ solutions of the focusing log energy-supercritical Schr\"odinger equation
$ i \partial_{t} u + \triangle u = -|u|^{\frac{4}{n-2}} u \log^{\gamma} ( 2  + |u|^{2})$ for a range of positive $\gamma \, s$ depending on the size of the initial data, for critical energies below the ground states', and for critical potential energies below. In order to control the barely supercritical nonlinearity in the virial identity and in the estimate of the growth of the critical energy for nonsmooth solutions, i.e solutions with data in $\tilde{H}^{k}$, $k \leq \frac{n}{2}$, we prove some Jensen-type inequalities, in the spirit of \cite{triroyjensen}.
\end{abstract}

\maketitle

\section{Introduction}

We shall study the radial solutions of the following focusing \footnote{It is well-known that the minus sign in the nonlinear term of
(\ref{Eqn:BarelySchrod}) makes the equation ``focusing''.} Schr\"odinger equation in dimension $n$, $n \in \{ 3,4,5 \}$: 

\begin{equation}
\begin{array}{ll}
i \partial_{t} u + \triangle u & =  -|u|^{\frac{4}{n-2}} u g(|u|) \cdot
\end{array}
\label{Eqn:BarelySchrod}
\end{equation}
Here $g(|u|):= \log^{\gamma} (2 + |u|^{2})$ and $\gamma > 0$ \footnote{
To simplify the exposition we have chosen in this paper to work with $g(|u|):= \log^{\gamma} (2 + |u|^{2})$. The reader can check that the main result  (i.e Theorem \ref{thm:main}) also holds for all functions $g$ of the form $g(|u|) := \log^{\gamma} ( b + |u|^{2}) $ with $b > 1$. See also Remark \ref{Rem:Famg}.}. Here $\log$ denotes the natural logarithm. \\
This equation has many connections with the following focusing power-type Schr\"odinger equation, $p>1$

\begin{equation}
\begin{array}{ll}
i \partial_{t} v + \triangle v & = -|v|^{p-1} v
\end{array}
\label{Eqn:Schrodpowerp}
\end{equation}
(\ref{Eqn:Schrodpowerp}) has a natural scaling: if $v$ is a solution of (\ref{Eqn:Schrodpowerp}) with data $v(0):=v_{0}$ and if $\lambda \in
\mathbb{R}$ is a parameter then $v_{\lambda}(t,x) := \frac{1}{\lambda^{\frac{2}{p-1}}} v \left( \frac{t}{\lambda^{2}}, \frac{x}{\lambda}
\right)$ is also a solution of (\ref{Eqn:Schrodpowerp}) but with data $v_{\lambda}(0,x):= \frac{1}{\lambda^{\frac{2}{p-1}}} v_{0} \left(
\frac{x}{\lambda} \right)$. If $s_{p}:= \frac{n}{2}- \frac{2}{p-1}$ then the $\dot{H}^{s_{p}}$ norm of the initial data is invariant under the
scaling: this is why (\ref{Eqn:Schrodpowerp}) is said to be $\dot{H}^{s_{p}}$- critical. If $p=1 + \frac{4}{n-2}$ then (\ref{Eqn:Schrodpowerp})
is $\dot{H}^{1}-$ (or energy$-$) critical. The focusing energy-critical Schr\"odinger equation
\begin{equation}
\begin{array}{ll}
i \partial_{t} u + \triangle u & = -|u|^{\frac{4}{n-2}} u
\end{array}
\label{Eqn:EnergyCrit}
\end{equation}
has received a great deal of attention. Cazenave and Weissler \cite{cazweiss} proved the local well-posedness of (\ref{Eqn:EnergyCrit}): given
any $u(0)$ such that $\| u(0) \|_{\dot{H}^{1}} < \infty$ there exists, for some $t_{0}$ close to zero, a unique $ u \in \mathcal{C} ( [0,t_{0}],
\dot{H}^{1} ) \cap L_{t}^{ \frac{2(n+2)}{n-2}} L_{x}^{\frac{2(n+2)}{n-2}} ( [0,t_{0}] )$ satisfying (\ref{Eqn:EnergyCrit}) in the sense of
distributions

\begin{equation}
\begin{array}{ll}
u(t) & = e^{it \triangle} u(0) + i \int_{0}^{t} e^{i(t-t^{'}) \triangle} \left[  |u(t')|^{\frac{4}{n-2}} u(t') \right] \, dt^{'} \cdot
\end{array}
\label{Eqn:DistribSchrod}
\end{equation}
The next step is to understand the asymptotic behavior of the radial solutions of (\ref{Eqn:EnergyCrit}). It is well-known that (\ref{Eqn:EnergyCrit}) has a family of stationary solutions (the ground states)  $W_{\lambda,\theta}(x) := e^{i \theta}\frac{1}{\lambda^{\frac{n-2}{2}}} W \left( \frac{x}{\lambda} \right)$
that satisfy

\begin{equation}
\begin{array}{l}
\triangle W_{\lambda,\theta} + |W_{\lambda,\theta}|^{\frac{4}{n-2}} W_{\lambda,\theta} =0
\end{array}
\label{Eqn:W}
\end{equation}
with $\theta \in [0, 2 \pi)$ and $W$ defined by $W(x)  := \frac{1}{\left(1 + \frac{|x|^{2}}{n(n-2)} \right)^{\frac{n-2}{2}}}$. \\
\\
The asymptotic behavior of the solutions for critical energies below the ground states' has been studied in \cite{kenmer}. In particular global existence and scattering (i.e the linear asymptotic behavior) were proved for critical potential energies below the ground states'  \footnote{Strictly speaking, the result in \cite{kenmer} is stated for critical energies (resp. kinetic energies) below the critical energy (resp. the critical kinetic energy) of the ground states. A simple argument shows that the conditions are the same if the word "kinetic" is replaced with "potential": see Appendix $A$. For a definition of the critical energy, the critical kinetic energy, and
the critical potential energy we refer to (\ref{Eqn:CriticalEnergy}).}. The asymptotic behavior of the solutions
was studied in \cite{duymerle} for energies equal to the ground states' and in \cite{roynak} for energies slightly larger than it.

If $p > 1+ \frac{4}{n-2}$ then $s_{p} > 1$ and we are in the energy supercritical regime. Since for all $\epsilon > 0$ there exists $c_{\epsilon} > 0$ such that $ \left| |u|^{\frac{4}{n-2}} u \right| \lesssim \left| |u|^{\frac{4}{n-2}} u g(|u|) \right| \leq c_{\epsilon} \max{(1, | |u|^{\frac{4}{n-2}+ \epsilon} u | ) }$ then the nonlinearity of (\ref{Eqn:BarelySchrod}) is said to be barely supercritical.

In this paper we study the asymptotic behavior of $\tilde{H}^{k}$-solutions of (\ref{Eqn:BarelySchrod}) for $n \in \{ 3,4,5 \}$. First we recall a 
local-wellposedness result \footnote{Proposition \ref{Prop:LocalWell} was proved in \cite{triroyrad} for $n \in \{3,4 \}$ and $k > \frac{n}{2}$. Nevertheless
it is stated slightly differently: ``$\mathcal{B} \left( L_{t}^{\frac{2(n+2)}{n-2}} L_{x}^{\frac{2(n+2)}{n-2}}([0,T_{l}]) ; 2 \epsilon \right)$ ''  is replaced with 
``$L_{t}^{\frac{2(n+2)}{n-2}} L_{x}^{\frac{2(n+2)}{n-2}}([0,T_{l}])$'' in (\ref{Eqn:Spaceu}). It is not difficult to see that Proposition \ref{Prop:LocalWell} implies
that (\ref{Eqn:DistribSchrodg}) also holds if we take into account this modification.}:

\begin{prop}{ \cite{triroyrad,triroyjensen}}
Let $n \in \{ 3,4,5 \}$ and let $ 1 < k < \infty$ if $n \in \{ 3,4 \}$ and $ 1 < k < \frac{4}{3}$ if $n=5$. Let $M$  be such that
$\| u_0 \|_{\tilde{H}^{k}} \leq M$. Then there exists $\epsilon := \epsilon(M) > 0$ such that the following holds: if  $T_{l}$ is a number
such that if $ T_{l} >  0$ satisfies 

\begin{equation}
\begin{array}{ll}
\| e^{i t \triangle} u_{0}  \|_{L_{t}^{\frac{2(n+2)}{n-2}} L_{x}^{\frac{2(n+2)}{n-2}} ([0,T_{l}]) } & \leq \epsilon,
\end{array}
\label{Eqn:SmallLinProp}
\end{equation}
then there exists a unique

\begin{equation}
\begin{array}{l}
u \in \mathcal{C}([0,T_{l}], \tilde{H}^{k}) \cap L_{t}^{\frac{2(n+2)}{n}} D^{-1} L_{x}^{\frac{2(n+2)}{n}} ([0,T_{l}])
\cap L_{t}^{\frac{2(n+2)}{n}} D^{-k} L_{x}^{\frac{2(n+2)}{n}}([0,T_{l}]) \cap \mathcal{B} \left( L_{t}^{\frac{2(n+2)}{n-2}} L_{x}^{\frac{2(n+2)}{n-2}}([0,T_{l}]) ; 2 \epsilon \right)
\end{array}
\label{Eqn:Spaceu}
\end{equation}
such that

\begin{equation}
\begin{array}{l}
u(t)=e^{i t \triangle} u_{0} + i \int_{0}^{t} e^{i(t- t^{'}) \triangle} \left( |u(t^{'})|^{\frac{4}{n-2}} u(t^{'}) g(|u(t^{'})|) \right) \,
dt^{'}
\end{array}
\label{Eqn:DistribSchrodg}
\end{equation}
is satisfied in the sense of distributions. Here $D^{-\alpha} L^{r}$ is the space endowed with the
norm $\| f \|_{D^{-\alpha} L^{r}} := \| D^{\alpha} f \|_{L^{r}}$. Here $\mathcal{B}(X; r)$ denotes the (closed) ball centered at the origin with radius $r$ in the normed space $X$.
\label{Prop:LocalWell}
\end{prop}

\begin{rem}
A number $T_{l}$ that satisfies the smallness condition above is called a time of local existence.
\end{rem}

This allows by a standard procedure to define the notion of maximal time interval of existence $I_{max}:= (T_{-},T_{+})$, that is the union of all the open intervals $I$ containing $0$ such that there exist a (unique) solution $u \in \mathcal{C}( I, \tilde{H}^{k}) \cap  L_{t}^{\frac{2(n+2)}{n}} D^{-1} L_{x}^{\frac{2(n+2)}{n}} (I) \cap  L_{t}^{\frac{2(n+2)}{n}} D^{-k} L_{x}^{\frac{2(n+2)}{n}} (I) $ that satisfies (\ref{Eqn:DistribSchrodg}) for all $t \in I$.

\begin{rem}
In the sequel we denote by $\tilde{H}^{k}-$ solution of (\ref{Eqn:BarelySchrod}) a distribution constructed by the standard procedure that

\begin{itemize}
\item satisfies (\ref{Eqn:DistribSchrodg}) for some $ u_{0} \in \tilde{H}^{k} $ and for all $t \in I_{max}$
\item lies in $ \mathcal{C} ( I, \tilde{H}^{k} ) \cap  L_{t}^{\frac{2(n+2)}{n}} D^{-1} L_{x}^{\frac{2(n+2)}{n}} (I) \cap  L_{t}^{\frac{2(n+2)}{n}} D^{-k} L_{x}^{\frac{2(n+2)}{n}} (I) $ for all interval $I \subsetneq I_{max} $
\end{itemize}

\end{rem}

\begin{rem}
In the sequel we say that $u$ is an $\tilde{H}^{k}-$ solution of (\ref{Eqn:BarelySchrod}) on an interval $I$ if $u$ is an $\tilde{H}^{k}-$ solution of
 (\ref{Eqn:BarelySchrod}) and $I \subset I_{max}$.
\end{rem}
Then we recall the following proposition:

\begin{prop}{\cite{triroyrad,triroyjensen}}
If $|I_{max}|< \infty $ then

\begin{equation}
\begin{array}{ll}
\| u \|_{L_{t}^{\frac{2(n+2)}{n-2}} L_{x}^{\frac{2(n+2)}{n-2}} (I_{max})} & = \infty
\end{array}
\end{equation}
\label{Prop:GlobWellPosedCrit}
\end{prop}

\begin{rem}
Proposition \ref{Prop:LocalWell} and \ref{Prop:GlobWellPosedCrit} were proved in \cite{triroyrad,triroyjensen} for solutions of loglog supercritical defocusing equations, i.e solutions of $i \partial_{t} u + \triangle u = |u|^{\frac{4}{n-2}} u q(|u|)$ with
$q(|u|) :=\log^{\tilde{\gamma}} \log \left( 10 + |u|^{2} \right)$ and $\tilde{\gamma} > 0$. The same proof works for solutions of (\ref{Eqn:BarelySchrod}).
\label{Rem:Defoc}
\end{rem}
The above proposition may be viewed as a criterion to prove global existence for an $\tilde{H}^{k}-$ solution $u$ of (\ref{Eqn:BarelySchrod}) \ (i.e $I_{max} = (- \infty, \infty) $). Indeed, arguing by contradiction, if we can prove an \textit{a priori} bound of the form $ \| u \|_{L_{t}^{\frac{2(n+2)}{n-2}} L_{x}^{\frac{2(n+2)}{n-2}} ([-T,T]) } \leq f(T, \| u_{0} \|_{\tilde{H}^{k}}) $ for arbitrarily large time $T>0$ then $u$ exists for all time. In fact we shall prove that for some data the bound does not depend on time $T$, which will imply linear asymptotic behavior (i.e., scattering). \\
\\
We then define the critical energy, notion that appears in the statement of the main theorem. If $f \in \dot{H}^{1}$  then we define the critical energy $E_{cr}(f)$ \footnote{Here $1_{2}^{*}$ is the critical Sobolev exponent, i.e it satisfies
$\frac{1}{1_{2}^{*}} = \frac{1}{2} - \frac{1}{n}$. It is well-known that if $u$ is a solution of (\ref{Eqn:EnergyCrit}) with data in $\dot{H}^{1}$ then
$E_{cr}(u(t))$ is conserved in time, that is $E_{cr}(u(t)) = E_{cr}(u(0))$. Hence the terminology ``critical''.}

\begin{equation}
E_{cr}(f) := \frac{1}{2} \int_{\mathbb{R}^{n}} |\nabla f(x)|^{2} \; dx - \frac{1}{1_{2}^{*}} \int_{\mathbb{R}^{n}} |f(x)|^{1_{2}^{*}} \; dx \cdot
\label{Eqn:CriticalEnergy}
\end{equation}
The critical energy is made up of two terms: the critical kinetic energy  $ \left( \, \text{i.e} \, \frac{1}{2} \int_{\mathbb{R}^{n}} |\nabla f(x)|^{2} \; dx \right) $, and the critical potential energy $\left( \, \text{i.e} \,  - \frac{1}{1_{2}^{*}} \int_{\mathbb{R}^{n}} |f(x)|^{1_{2}^{*}} \; dx \right)$. \\
\vphantom{ab} \\
We now explain the main incentive for studying the asymptotic behavior of $\tilde{H}^{k}-$ solutions of (\ref{Eqn:BarelySchrod}). Defocusing barely supercritical
equations have been extensively in the literature: see e.g \cite{shihhsiwei,triroyrad,triroyjensen,triroysmooth,taolog}. In particular, global existence and
scattering of (some) solutions have been proved without any restriction regarding the size of the data by extending the theory for (some) solutions of the defocusing critical equations. In this paper we would like to know whether for radial data we can also extend the theory for (some) solutions of focusing critical Schrodinger equations to (some)
solutions of focusing barely supercritical Schrodinger equations. Here, unlike the defocusing equations, and by analogy with the theory developed in \cite{kenmer}
for (\ref{Eqn:EnergyCrit}), it is natural to consider data with size (in a sense to be defined) smaller than the ground states'
for the $\tilde{H}^{k}-$ solutions of (\ref{Eqn:BarelySchrod}) to scatter, since the ground states are global-in-time but non-scattering solutions of (\ref{Eqn:EnergyCrit}), an equation that is closely related to (\ref{Eqn:BarelySchrod}). More precisely the main result of this paper is a global existence and scattering result for a barely supercritical nonlinearity, namely a focusing size-dependent \footnote{Observe that $\gamma$ satisfies the condition (\ref{Eqn:CdtRes}). This condition depends on the size of the data. Hence the terminology ``size-dependent''.} log energy-supercritical Schr\"odinger equation for critical energies below the ground states' and for critical potential energies below:

\begin{thm}
Let $n \in \{3,4,5 \}$. \\
Let $I_n$ defined as follows: if $n \in \{3,4\}$ then $I_n := (1,\infty)$ and if
$n=5$ then  $I_n := \left( 1, \frac{4}{3} \right)$. Let $k \in I_{n}$. Let $\delta > 0$. Let $ u_{0} \in \tilde{H}^{k} $ be such that

\begin{equation}
E_{cr}(u_0) < (1- \delta) E_{cr}(W), \; \text{and} \;\| u_{0} \|_{L^{1_{2}^{*}}} < \| W \|_{L^{1_{2}^{*}}} \cdot
\label{Eqn:Ass0}
\end{equation}
Then there exists a large constant $ \bar{C} > 0 $ such that if $\gamma > 0$ satisfies the condition $(Cd)$ below

\begin{equation}
(Cd): \left\{
\begin{array}{ll}
k-1 \leq 1: & (k - 1)^{-\gamma} \log^{\gamma} \left( \bar{C}^{\bar{C}^{\bar{C}^{\delta^{-\frac{1}{2}}}}}  \| u_{0} \|_{\tilde{H}^{k}} \right)
-1 \leq  \bar{C}^{-1} \delta \\
k - 1 \geq 1: &  \log^{\gamma} \left( \bar{C}^{\bar{C}^{\bar{C}^{\delta^{-\frac{1}{2}}}}}  \| u_{0} \|_{\tilde{H}^{k}} \right) - 1  \leq  \bar{C}^{-1} \delta,
\end{array}
\right.
\label{Eqn:CdtRes}
\end{equation}
then $u$, the $\tilde{H}^{k}-$ solution of (\ref{Eqn:BarelySchrod}) with radial data $ u_{0} \in \tilde{H}^{k}$, exists for all time. Moreover
there exists a scattering state $u_{0,\pm} \in \tilde{H}^{k} $ such that

\begin{equation}
\begin{array}{ll}
\lim \limits_{t \rightarrow \pm \infty} \| u(t) - e^{it \triangle} u_{0,\pm}  \|_{\tilde{H}^{k}} & = 0 \cdot
\end{array}
\label{Eqn:Scatt}
\end{equation}

\label{thm:main}
\end{thm}


\vphantom{ab} \\

\begin{rem}
If we only assume $ \| u_{0} \|_{\tilde{H}^{k}} \ll 1 $ \footnote{ \label{footsymbol}
For a definition of the $\ll$ symbol and the $\gtrsim$ symbol we refer to
Section \ref{Sec:Prelim}.} then the same conclusion holds for all $\gamma > 0$: this is a consequence of the local theory:
see Appendix $B$.
\end{rem}

Consequently we make the following assumption: \\
\\
\underline{Assumption}: we will assume from now on that $ \| u_{0} \|_{\tilde{H}^{k}} \gtrsim 1 $ \footnoteref{footsymbol}.

\begin{rem}
The assumptions imply that $\delta < \frac{n}{2} $ \footnote{Indeed this follows from (\ref{Eqn:Ass0}) and (\ref{Eqn:ValKW}).}.
Consequently we see from (\ref{Eqn:CdtRes}) that $\gamma$ is small. (\ref{Eqn:CdtRes}) says how small $\gamma$ \footnote{The smallness depends on the value of
$\| u_{0} \|_{\tilde{H}^{k}}$ and on the value of $\delta$. } should be for the solution $u$ of (\ref{Eqn:BarelySchrod}) to scatter in the regime where $k$ is close to $1$ ( i.e $k-1 \leq 1$) and in the regime where $k$ is far from $1$ ( i.e $k-1 \geq 1$).
\label{Rem:SmallGamma}
\end{rem}

\begin{rem}
Observe that if we only assume that $E_{cr}(u_{0}) < (1 - \delta) E_{cr}(W) $ and that  $\gamma > 0$ satisfies (\ref{Eqn:CdtRes}) then $ \| u_{0} \|_{L^{1_{2}^{*}}} = \| W \|_{L^{1_{2}^{*}}} $ is impossible. \\
\vphantom{ab} \\
Indeed recall  see (e.g \cite{aubin,talenti}) the sharp Sobolev inequality

\begin{equation}
\| f \|_{L^{1_{2}^{*}}} \leq C_{*} \| \nabla f \|_{L^{2}}, \; \text{with} \;  C_{*} := \frac{\| W \|_{L^{1_{2}^{*}}}}{\| \nabla W \|_{L^{2}}} \cdot
\nonumber
\end{equation}
Hence $F(y) < \left( 1 -  \frac{\delta}{2} \right) F(\| W \|_{L^{1_{2}^{*}}}) $ with $y:= \| u_{0} \|_{L^{1_{2}^{*}}}$ and
$ F(y) := \frac{1}{2 C_{*}^{2}} y^{2} - \frac{1}{1_{2}^{*}} y^{1_{2}^{*}} $. \\
\vphantom{ab} \\
So it remains to understand the asymptotic behavior of  $\tilde{H}^{k}-$ solutions of (\ref{Eqn:BarelySchrod}) if $u_{0} \in \tilde{H}^{k}$ satisfies $E_{cr}(u_{0}) < (1 -  \delta) E_{cr} (W)$, $\| u_{0} \|_{L^{1_{2}^{*}}} > \| W \|_{L^{1_{2}^{*}}} $, and if $\gamma > 0$ satisfies (\ref{Eqn:CdtRes}) with $\bar{C} > 0$ a constant that is large enough. We will not pursue this matter in this paper.
\label{Rem:Imp}
\end{rem}

\begin{rem}
In this paper we have chosen to work with $g(|u|):= \log^{\gamma} (2 + |u|^{2})$ to simplify the exposition. The reader can check that a straightforward modification of the arguments in this paper shows that for all $b > 1 $ Theorem \ref{thm:main} also holds for $g(|u|) := \log^{\gamma} ( b + |u|^{2}) $ with $\bar{C}:= \bar{C}(b)$.
\label{Rem:Famg}
\end{rem}

Now we explain the main ideas of this paper and how it is organized. In Section \ref{Sec:Thmmain} we prove the main result of this paper, i.e Theorem \ref{thm:main}. The
proof relies upon the following estimate of $\| u \|_{L_{t}^{\frac{2(n+2)}{n-2}} L_{x}^{\frac{2(n+2)}{n-2}} }$ on an arbitrarily long-time interval $J$, assuming that
an \textit{a priori} bound of some norms at $\tilde{H}^{k}-$ regularity and a smallness condition hold
\footnote{For a definition of $Q(J,u)$ and $\tilde{h}$ we refer to Section \ref{Sec:Prelim} and to Section \ref{Sec:ProofProp}.}:

\begin{prop}
Let $u$ be a radial $\tilde{H}^{k}-$ solution of (\ref{Eqn:BarelySchrod}). Let
$J:=[0,a]$ be an interval. There exist $ C_{0} \gg 1$ and $M_{0} \gg 1 $ such that if

\begin{equation}
\begin{array}{l}
M \geq M_{0}: \; Q(J,u) \leq M, \; \text{and} \; \tilde{h} \left(C_{0}  M \right) \ll \delta,
\end{array}
\label{Eqn:Constg}
\end{equation}
then

\begin{equation}
\begin{array}{ll}
\| u \|^{\frac{2(n+2)}{n-2}}_{ L_{t}^{\frac{2(n+2)}{n-2}} L_{x}^{\frac{2(n+2)}{n-2}} (J)} & \leq
C_{0}^{C_{0}^{ \delta^{-\frac{1}{2}}}} \cdot
\end{array}
\label{Eqn:BoundLong}
\end{equation}
\label{Prop:BoundLong}
\end{prop}
This bound proved on an arbitrary time interval $J$, combined with a local induction in time of some Strichartz estimates and the smallness of $\gamma$ (see
(\ref{Eqn:CdtRes})), allows to show \textit{a posteriori} that the condition (\ref{Eqn:Constg}) holds on $J$: see \cite{triroyrad, triroysmooth} for a similar argument. Global well-posedness and scattering of $\tilde{H}^{k}$-solutions of (\ref{Eqn:BarelySchrod}) follow easily from the finiteness of these bounds. In Section \ref{Sec:ProofProp}, we prove Proposition \ref{Prop:BoundLong}. We mention the main differences between this paper
and the papers \cite{triroyrad,triroyjensen}. First one has to assume the smallness condition in (\ref{Eqn:Constg}). This condition combined with the energy conservation law and the properties of the ground states assure that some relevant norms (such as the kinetic energy and the potential energy) are bounded on
$J$, so that we can apply the techniques of concentration (see, e.g., \cite{bourg,taorad}) in order to prove
(\ref{Eqn:BoundLong}). Roughly speaking, we divide $J$ into subintervals $(J_{l})_{ 1 \leq l \leq L}$ such that the $ L_{t}^{\frac{2(n+2)}{n-2}} L_{x}^{\frac{2(n+2)}{n-2}} $ norm of $u$ concentrates, i.e, it is small but also substantial. Our goal is to estimate the number of these subintervals. It is already known that the mass on a ball centered at the origin concentrates for all time on each of these subintervals. In \cite{triroyrad}, a Morawetz-type estimate (combined with the mass concentration) was used to prove that the following statement holds: one of these subintervals is large compare with $J$. In this paper we prove a decay at some time of the potential energy on a ball
centered at the origin by using the virial identity, which leads to a contradiction unless the statement above holds. When we use the
virial identity, one has to control some error terms. One also has to control the growth of the critical energy. In order to achieve these goals and close the argument
at the final stage of the proof (see Section \ref{Sec:Thmmain}) for nonsmooth solutions, we prove some Jensen-type inequalities (in the spirit of \cite{triroyjensen}) with respect to well-chosen measures ( and within a contradiction argument, if necessary) and we adapt arguments in \cite{kenmer,roynak,taofoc} to prove the decay. Once the statement is proved one can show that there exists a significant number of subintervals (in comparison with the total number of subintervals) that concentrate around some time and such that the mass concentrates around the origin, which yields an estimate of the number of all the subintervals. The process involves several estimates. One has to understand how they depend on $\delta$ since this will play an important role in the choice of $\gamma$ (see (\ref{Eqn:CdtRes})) for which we have global well-posedness and scattering of radial $ \tilde{H}^{k} $-solutions of (\ref{Eqn:BarelySchrod}). \\
\\
\textbf{Funding}: This article was partially funded by an URB grant (ID: $3245$) from the American University of Beirut. This research was also partially conducted in Japan and
funded by a (JSPS) Kakenhi grant  [ 15K17570 to T.R.].

\section{Preliminaries}
\label{Sec:Prelim}

\subsection{Notation}

We recall some notation. If $a \in \mathbb{R}$ then $\langle a \rangle := \left( 1+ a^{2} \right)^{\frac{1}{2}}$. We write $a \ll b$ if the value of $a$ is much smaller that that of $b$, $a \gg b$ if the value of $a$ is much larger than that of $b$, and $a \approx b $ if  $ a \ll b $ and $b \ll a$ are not true. We write
$a = o(b)$ if there exists a constant $0 < c \ll 1$ such that  $|a| \leq c |b|$. We define
$b+ = b + \epsilon$ for $ 0 < \epsilon \ll 1$. If $b+$ appears in a mathematical expression such as
$a \leq C b+$, then we ignore the dependance of $C$ on $\epsilon$ in order to make our presentation simple.
Unless otherwise specified, we let in the sequel $f$ (resp. $u$) be a function depending on the space variable (resp.
the space variable and the time variable). Unless otherwise specified, for sake of simplicity, we do not mention the spaces to which $f$ and $u$
belong in the estimates: this exercise is left to the reader. \\
\\
Let $r > 1$ and let $m$ be a positive number such that $m < \frac{n}{r}$. We denote by $m_{r}^{*}$ the number that satisfies
$\frac{1}{m_{r}^{*}} = \frac{1}{r} - \frac{m}{n}$. \\
\\
Let $p \geq 1$. Let $J$ be an interval such that $J \subset \mathbb{R}$ . We define the number

\begin{equation}
\begin{array}{l}
Q_{p}(J,u) :=  \| u \|_{L_{t}^{\infty} \tilde{H}^{p}(J)} + \| D u \|_{L_{t}^{\frac{2(n+2)}{n}} L_{x}^{\frac{2(n+2)}{n}} (J)} + \| D^{p} u
\|_{L_{t}^{\frac{2(n+2)}{n}} L_{x}^{\frac{2(n+2)}{n}} (J) } \cdot
\end{array}
\nonumber
\end{equation}
If $p=k$ then we omit the `$k$' for sake of simplicity and we write $Q(J,u)$ for $Q_{k}(J,u)$.

\subsection{Standard estimates}

In this subsection we write down some standard estimates. \\
\\
We recall (a generalized form of) the Jensen inequality that we will often use throughout this paper:

\begin{prop}(Jensen inequality, see e.g \cite{yeh})
Let $(X, \mathcal{B}, \mu)$ be a measure space such that $0 < \mu(X) < \infty$. Let $I$ be an open interval and let
$g$ be a $\mu-$ integrable function on a set $D \in \mathcal{B}$ such that $g(D) \subset I$. If $f$ is a concave function on
$I$ then the following holds

\begin{equation}
\begin{array}{ll}
\frac{1}{\mu(D)} \int_{D} f \circ  g \; d \mu &  \leq f \left( \frac{ \int_{D} g \; d \mu }{\mu(D)} \right)
\end{array}
\nonumber
\end{equation}

\end{prop}
The statement of this inequality is made in \cite{yeh} for convex functions. The statement of the inequality below follows immediately from that in \cite{yeh} , taking into account that if $f$ is concave then $-f$ is convex.  \\
\\
We then recall some Strichartz-type estimates. Let $J$ be an interval. Let $t_0 \in J$. If $u$ is a solution of $i \partial_{t} u + \triangle u = G$ on $J$ then the Strichartz estimates (see e.g \cite{keeltao}) yield

\begin{equation}
\begin{array}{l}
\| u \|_{L_{t}^{\infty} \dot{H}^{j} (J) } + \| D^{j} u \|_{L_{t}^{\frac{2(n+2)}{n}} L_{x}^{\frac{2(n+2)}{n}} (J)} + \| D^{j} u
\|_{L_{t}^{\frac{2(n+2)}{n-2}} L_{x}^{\frac{2 n(n+2)}{n^{2}+4}} (J)} \\
\lesssim \| D^{j} G \|_{L_{t}^{\frac{2(n+2)}{n+4}} L_{x}^{\frac{2(n+2)}{n+4}} ( J ) } + \| u(t_0) \|_{\dot{H}^{j}}
\end{array}
\label{Eqn:Strich}
\end{equation}
if $j \geq 1$. We write

\begin{equation}
\begin{array}{l}
u(t) = u_{l,t_0}(t)+ u_{nl,t_0}(t)
\end{array}
\nonumber
\end{equation}
with $u_{l,t_0}$ denoting the linear part starting from $t_0$, i.e

\begin{equation}
\begin{array}{l}
u_{l,t_0}(t) := e^{i(t-t_0) \triangle} u(t_0),
\end{array}
\nonumber
\end{equation}
and $u_{nl,t_0}$ denoting the nonlinear part starting from $t_0$, i.e

\begin{equation}
\begin{array}{l}
u_{nl,t_0}(t) :=  - i \int_{t_0}^{t} e^{i(t-s) \triangle} G(s) \; ds \cdot
\end{array}
\nonumber
\end{equation}

We then recall some properties for $u$, an $\tilde{H}^{k}-$ solution of (\ref{Eqn:BarelySchrod}). \\
Let $E_{\gamma}(f)$ denotes the energy of a function $f \in \tilde{H}^{k}$, that is

\begin{equation}
\begin{array}{ll}
E_{\gamma}(f)  & := \frac{1}{2} \int_{\mathbb{R}^{n}} |\nabla f(x)|^{2} \, dx -  \int_{\mathbb{R}^{n}} P(|f|)(x) \, dx,
\end{array}
\label{Eqn:EnergyBarely}
\end{equation}
with  $P(|z|) := \int_{0}^{|z|} s^{1_2^{*}-1 } g(s) \, ds$.  Observe that the energy is finite \footnote{Indeed, taking into account $g(|f|) \lesssim 1 + |f|^{k_{2}^{*}-1_{2}^{*}}$, the embeddings $\tilde{H}^{k} \hookrightarrow L^{1_{2}^{*}}$ and $\tilde{H}^{k} \hookrightarrow L^{k_{2}^{*}}$, we get $\left| \int_{\mathbb{R}^{n}} P(|f|)(x) \, dx \right| \lesssim  \| f \|^{1_{2}^{*}}_{\tilde{H}^{k}} + \| f \|^{k_{2}^{*}}_{\tilde{H}^{k}} < \infty \cdot $}. Then it is well-known that $u$ satisfies the conservation of the energy, that is $E_{\gamma}(u(t)) = E_{\gamma} (u(0)) $ for all $t \in I_{max}$ \footnote{More precisely, if $n \in \{ 3,4 \}$ then a computation shows that $E_{\gamma}(u(t))= E_{\gamma}(u(0))$ for smooth solutions (i.e, solutions in $\tilde{H}^{p}$ with exponents $p$ large enough). Then $E_{\gamma}(u(t)) = E_{\gamma}(u(0))$  holds for an $\tilde{H}^{k}-$ solution by a standard approximation with smooth solutions. If $n=5$ then the nonlinearity is not even smooth. So one should first smooth out the nonlinearity and obtain an identity similar to the energy identity for smooth solutions (i.e, solutions lying in Sobolev spaces with large exponent)
of the ``smoothed'' equation and then take limit in $\tilde{H}^{k}$ by a standard approximation with smooth solutions. \label{Foot:EnergyRem}}.\\
Recall (see, e.g, \cite{taorad}) that

\begin{equation}
\begin{array}{ll}
Mass \left( u(t),B(x_{0},R) \right) & \lesssim  R \, \| \nabla u(t) \|_{L^{2}},
\end{array}
\label{Eqn:MassControl}
\end{equation}
with $ Mass \left( u(t), B(x_{0},R) \right)$ denoting the mass of $u(t)$ within the ball $B(x_{0},R)$ centered at $x_{0}$ and with radius
$R  > 0$, i.e

\begin{equation}
\begin{array}{ll}
Mass \left( u(t), B(x_{0},R) \right) & := \left( \int_{\mathbb{R}^{n}} \chi \left( \frac{|x-x_{0}|}{R} \right) |u(t,x)|^{2} \; dx \right)^{\frac{1}{2}} \cdot
\end{array}
\nonumber
\end{equation}
Here $\chi$ is a smooth function such that $\chi(x)= 1 $ if $|x| \leq 1$ and $\chi(x)=0$ if $|x| \geq 2$. Recall that the derivative satisfies (see also, e.g., \cite{taorad})

\begin{equation}
\begin{array}{ll}
\left| \partial_{t} Mass(u(t),B(x_{0},R)) \right| & \lesssim \frac{ \|\nabla u(t) \|_{L^{2}}}{R}
\end{array}
\label{Eqn:UpBdDerivM}
\end{equation}

\subsection{Leibnitz-type rules}

We now turn our attention to Leibnitz-type rules. Leibnitz-type rules will be used to control the nonlinearity on small intervals $J$. \\
We recall the following proposition:

\begin{prop}{\cite{triroyjensen}}
Let $2 \geq k^{'} > 1 $. Let $j \in \{ 1,k' \}$. Let $J$ be an interval such that $\| u \|_{L_{t}^{\frac{2(n+2)}{n-2}} L_{x}^{\frac{2(n+2)}{n-2}}(J)} \leq P \lesssim 1$. Then

\begin{equation}
\begin{array}{ll}
\left\| D^{j} \left( |u|^{\frac{4}{n-2}} u g(|u|) \right) \right\|_{L_{t}^{\frac{2(n+2)}{n+4}}  L_{x}^{\frac{2(n+2)}{n+4}} (J)}
& \lesssim \| D^{j} u \|_{L_{t}^{\frac{2(n+2)}{n}} L_{x}^{\frac{2(n+2)}{n}} (J)} P^{\frac{4}{n-2} -} g(Q_{k^{'}}(J,u))
\end{array}
\label{Eqn:EstDj}
\end{equation}
\label{Prop:Nonlin}
\end{prop}
This proposition combined with (\ref{Eqn:Strich}) will allow to control for $ 1 < k^{'} \leq 2 $ quantities $Q_{k'}(J,u)$
that contain norms at $\dot{H}^{k'}-$ regularity  ( such as $ \| D^{k'} u \|_{L_{t}^{\frac{2(n+2)}{n}} L_{x}^{\frac{2(n+2)}{n}} (J) } $ )
and norms at $\dot{H}^{1}-$ regularity  ( such as  $ \| \nabla u \|_{L_{t}^{\frac{2(n+2)}{n}} L_{x}^{\frac{2(n+2)}{n}} (J) } $ ). \\
The following proposition holds ( see \cite{triroyrad} and Remark \ref{Rem:NonlinFracSmooth} ):

\begin{prop}
Let $ 0  \leq \alpha < 1$, $k'$ and $\beta$ be integers such that $k' \geq 2$ and $\beta \geq k'-1 $, $(r ,r_{2}) \in (1,\infty)^{2}$,
$(r_{1},r_{3}) \in (1, \infty]^{2}$ be such that $\frac{1}{r}= \frac{\beta}{r_{1}} + \frac{1}{r_{2}} +\frac{1}{r_{3}}$. Let $F: \mathbb{R}^{+} \rightarrow \mathbb{R}$
be a $\mathcal{C}^{k'}-$ function, let $\tilde{F} : \mathbb{R}^{+} \rightarrow \mathbb{R} $ be a nondecreasing
$\mathcal{C}^{0}-$ function, and let $ G := \mathbb{R}^{2} \rightarrow \mathbb{R}^{2} $ be a $ \mathcal{C}^{k^{'}-1} -$  function that lies also in $ \mathcal{C}^{k^{'}} \left(  \mathbb{R}^{2} - \{ (0,0) \} \right)$. Assume that

\begin{equation}
\begin{array}{l}
(a): \;
\left\{
\begin{array}{l}
\tau \in [0,1]: \; \left| F \left( |\tau x + (1-\tau)y|^{2} \right) \right|
\lesssim  \left| \tilde{F}(|x|^{2}) \right| + \left| \tilde{F}( |y|^{2}) \right|,   \\
\tilde{F}(x) \gtrsim \tilde{F}(0), \; \text{and} \\
i \in \{0,...,k^{'} \}:  \; F^{[i]}(x) = O \left( \frac{\tilde{F}(x)}{x^{i}} \right)
\end{array}
\right.
\\
(b): \; i \in \{0,...,k^{'} \}: \;  G^{[i]}(x,\bar{x}) = O \left( |x|^{\beta + 1  - i} \right) \cdot
\end{array}
\nonumber
\end{equation}
Then

\begin{equation}
\begin{array}{ll}
\left\| D^{ k' -1 + \alpha} \left( G(f,\bar{f}) F(|f|^{2}) \right) \right\|_{L^{r}} & \lesssim \| f \|^{\beta}_{L^{r_{1}}} \| D^{k' -1  + \alpha} f \|_{L^{r_{2}}}
\| \tilde{F}(|f|^{2}) \|_{L^{r_{3}}}
\end{array}
\label{Eqn:EstToProveFrac}
\end{equation}
Here $F^{[i]}$ and $G^{[i]}$ denotes the $i^{th}-$ derivative of $F$ and $G$ respectively.

\label{Prop:NonlinFracSmooth}
\end{prop}

\begin{rem}
In Appendix $C$, we explain why Proposition \ref{Prop:NonlinFracSmooth} for $ \beta > k' - 1 $ was proved in the proof of
Proposition $7$ in \cite{triroyrad}. 
In Appendix $D$, we prove Proposition \ref{Prop:NonlinFracSmooth} for $ \beta = k' - 1 $.
\label{Rem:NonlinFracSmooth}
\end{rem}

Observe that if $ n \in \{ 3,4 \} $ , $ G(z,\bar{z}) := |z|^{\frac{4}{n-2}} z $, $F(x):= \log^{\gamma} (2 + x)$, and
$\tilde{F}(x):= F(x)$, then $G$, $F$, and $\tilde{F}$ satisfy  the assumptions of Proposition \ref{Prop:NonlinFracSmooth}  with $ \beta := \frac{4}{n-2}$
and $k' \in \mathbb{N}$ such that $ \frac{n+2}{n-2} \geq k' \geq 2 $. Hence by combining this proposition with (\ref{Eqn:Strich}) and (Sobolev) embeddings we will be able to control for this range of $k's$ quantities such as $Q_{k'}(J,u)$ that contain norms at $\dot{H}^{k'-1 + \alpha}-$ regularity. \\
Hence, by applying Proposition \ref{Prop:Nonlin} and Proposition \ref{Prop:NonlinFracSmooth} to $J$ interval such that
$\| u \|_{L_{t}^{\frac{2(n+2)}{n-2}} L_{x}^{\frac{2(n+2)}{n-2}} (J)} \lesssim 1$, by the embedding
$ \tilde{H}^{\frac{n}{2}+} \hookrightarrow L^{\infty}$, and by the estimate $k' > \frac{n}{2}: \; Q_{\frac{n}{2}+} (J,u) \lesssim  Q_{k'}(J,u)$
that follows from interpolation, we get for $  k^{'} \in  \left( 1 , \frac{n+2}{n-2} \right) \cap I_{n} $ and for $j \in \{ 1,k' \}$
\footnote{Divide into three cases: $ \frac{n+2}{n-2} > k^{'} > \frac{n}{2} $ if $ n \in \{ 3,4 \} $, $  \frac{n}{2}  \geq  k^{'} > 1 $ if $n \in \{3,4 \}$, and
$ \frac{4}{3} > k^{'} > 1 $ if $n= 5$.}

\begin{equation}
\begin{array}{ll}
\left\| D^{j} \left( |u|^{\frac{4}{n-2}} u g(|u|) \right) \right\|_{L_{t}^{\frac{2(n+2)}{n+4}} L_{x}^{\frac{2(n+2)}{n+4}} (J)} & \lesssim
\| D^{j} u \|_{L_{t}^{\frac{2(n+2)}{n}} L_{x}^{\frac{2(n+2)}{n}} (J) }
\| u \|^{\frac{4}{n-2}-}_{L_{t}^{\frac{2(n+2)}{n-2}} L_{x}^{\frac{2(n+2)}{n-2}} (J) }
g \left( Q_{k^{'}}(J,u) \right) \cdot
\end{array}
\label{Eqn:NonlinEst}
\end{equation}
This estimate will be used in Section \ref{Sec:Thmmain} and in Section \ref{Sec:ProofProp}. We also recall the following proposition that only holds for $n \in \{ 3, 4 \}$:

\begin{prop}{\cite{triroyrad}}
Let $(\lambda_1,\lambda_2) \in \mathbb{N}^{2}$ be such that $\lambda_1 + \lambda_2 = \frac{n+2}{n-2}$. Let $J$ be an interval. Let $k^{'} > \frac{n}{2}$. Then there exists
a constant $C  > 0 $ such that

\begin{equation}
\begin{array}{ll}
\left\| D^{k^{'}}(u^{\lambda_1} \bar{u}^{\lambda_2} g(|u|) ) \right\|_{L_{t}^{\frac{2(n+2)}{n+4}} L_{x}^{\frac{2(n+2)}{n+4}}(J)} & \lesssim
\| u \|^{\frac{4}{n-2}}_{L_{t}^{\frac{2(n+2)}{n-2}} L_{x}^{\frac{2(n+2)}{n-2}}(J)} \langle  Q_{k^{'}}(J,u) \rangle^{C}
\end{array}
\label{Eqn:EstHighReg}
\end{equation}
\label{Prop:EstHighReg}
\end{prop}
This proposition will allow to control for $  k^{'} > \frac{n}{2} $ norms at
$\dot{H}^{k'}-$ regularity of the nonlinearity such as
$ \left\| D^{k^{'}} \left( |u|^{\frac{4}{n-2}} u g(|u|) \right) \right\|_{L_{t}^{\frac{2(n+2)}{n+4}} L_{x}^{\frac{2(n+2)}{n+4}}(J)}$: see Section \ref{Sec:Thmmain}.

\section{Proof of Theorem \ref{thm:main}}   
\label{Sec:Thmmain}
The proof is made of three steps:

\begin{itemize}

\item given $k \in I_{n} \cap \left( 1, \frac{n+2}{n-2} \right)$, finite bound for all $T \geq 0$ of $Q([-T,T],u)$: this implies a finite
bound of $\| u \|_{L_{t}^{\frac{2(n+2)}{n-2}} L_{x}^{\frac{2(n+2)}{n-2}} ([-T,T])}$ in view of the embedding
$ \| f \|_{L^{\frac{2(n+2)}{n-2}}} \lesssim \| D f \|_{L^{\frac{2n(n+2)}{n^{2}+4}}} $ and
the interpolation inequality below that holds for $\theta$ such that $\frac{n-2}{2(n+2)} = \frac{n(1- \theta) }{2(n+2)} $:

\begin{equation}
\begin{array}{l}
\| Du \|_{L_{t}^{\frac{2(n+2)}{n-2}} L_{x}^{\frac{2n(n+2)}{n^{2}+4}} ([-T,T]) } \lesssim \| D u \|^{\theta}_{L_{t}^{\infty} L_{x}^{2} ([-T,T])} \| D u \|^{1- \theta}_{L_{t}^{\frac{2(n+2)}{n}} L_{x}^{\frac{2(n+2)}{n}} ([-T,T]) }.
\end{array}
\nonumber
\end{equation}
This implies global existence by Proposition \ref{Prop:GlobWellPosedCrit} and the paragraph below Remark \ref{Rem:Defoc}. In fact we shall prove that the bound does not depend on $T$, which will imply scattering: see third step. By symmetry \footnote{i.e if $t \rightarrow u(t,x)$ is a solution of (\ref{Eqn:BarelySchrod}) then
$ t \rightarrow \bar{u}(-t,x) $ is a solution of (\ref{Eqn:BarelySchrod}).} it is sufficient to restrict our analysis to the interval $[0,T]$. \\
Let $ C_{1} \gg 1 $  be a constant large enough such that the statements below are true. Let $ M_k  := C_{1}^{C_{1}^{C_{1}^{\delta^{-\frac{1}{2}}}}} \| u_0 \|_{\tilde{H}^{k}}$. We claim that for all $ T > 0 $ we have $Q([0,T],u) \leq M_k$. If not there exists $\bar{T} > 0 $ such that $Q([0,\bar{T}],u) = M_k$ and for
all $t \in [0,\bar{T}]$, we have $Q([0,t]) \leq M_k$. The properties of $k_{\diamond}$ (see Section \ref{Sec:ProofProp}) and Remark
\ref{Rem:SmallGamma} imply that the second estimate of (\ref{Eqn:Constg}) and $g(M_{k}) \lesssim 1$ hold. Applying Proposition \ref{Prop:BoundLong} we see that

\begin{equation}
\begin{array}{ll}
\| u \|^{\frac{2(n+2)}{n-2}}_{L_{t}^{\frac{2(n+2)}{n-2}} L_{x}^{\frac{2(n+2)}{n-2}} ([0,\bar{T}])} \leq C_{0}^{C_{0}^{\delta^{- \frac{1}{2}}}}
\end{array}
\label{Eqn:StrichBound}
\end{equation}
Let $J:=[0,a]$ be an interval such that $ \| u \|_{L_{t}^{\frac{2(n+2)}{n-2}} L_{x}^{\frac{2(n+2)}{n-2}} (J)} \lesssim 1$. Let
$0 < \epsilon \ll 1$. From (\ref{Eqn:Strich}) and (\ref{Eqn:NonlinEst}), we see that
there is a constant $C \gtrsim 1$ such that

\begin{equation}
\begin{array}{ll}
Q(J,u) & \lesssim \| u_{0} \|_{\tilde{H}^{k}} + \left(  \| D u \|_{L_{t}^{\frac{2(n+2)}{n}} L_{x}^{\frac{2(n+2)}{n}} (J) } + \| D^{k} u
\|_{L_{t}^{\frac{2(n+2)}{n}} L_{x}^{\frac{2(n+2)}{n}} (J) } \right) \| u \|^{\frac{4}{n-2}-}_{L_{t}^{\frac{2(n+2)}{n-2}} L_{x}^{\frac{2(n+2)}{n-2}} (J) }
g \left(  Q(J,u)\right)
\\
& \leq C \| u_{0} \|_{\tilde{H}^{k}} + C  Q(J,u) \| u \| ^{\frac{4}{n-2}-}_{L_{t}^{\frac{2(n+2)}{n-2}} L_{x}^{\frac{2(n+2)}{n-2}} (J)} \cdot
\end{array}
\label{Eqn:ModelStrichartz}
\end{equation}
Let $0 < \epsilon  \ll 1$. Notice that if $J$ satisfies $ \| u \|_{L_{t}^{\frac{2(n+2)}{n-2}} L_{x}^{\frac{2(n+2)}{n-2}} (J)} \leq \epsilon  $, then a simple continuity argument shows that $Q(J,u) \leq 2 C \| u_{0} \|_{\tilde{H}^{k}}$. \\
In view of (\ref{Eqn:StrichBound}) we can divide $[0, \bar{T}]$ into subintervals
$(J_{i})_{1 \leq i \leq I}$ such that $ \| u \|_{L_{t}^{\frac{2(n+2)}{n-2}} L_{x}^{\frac{2(n+2)}{n-2}}
(J_{i})} = \epsilon $, $ 1 \leq i < I$ and
$ \| u \|_{L_{t}^{\frac{2(n+2)}{n-2}} L_{x}^{\frac{2(n+2)}{n-2} } (J_{I})} \leq \epsilon $,  with

\begin{equation}
\begin{array}{ll}
I & \lesssim C_{0}^{C_{0}^{\delta^{- \frac{1}{2}}}}  \cdot
\end{array}
\nonumber
\end{equation}
We get by iteration (increasing the value of $C_0$ if necessary)

\begin{equation}
\begin{array}{l}
Q \left( [0,\bar{T}],u \right)  \lesssim (2C)^{I} \| u_0 \|_{\tilde{H}^{k}}  \leq  C_{0}^{C_{0}^{C_{0}^{\delta^{- \frac{1}{2}}}}  }  \| u_0 \|_{\tilde{H}^{k}}
\leq \frac{M_k}{2} \cdot
\end{array}
\nonumber
\end{equation}
This  is a contradiction.

\item Finite bound of $Q(\mathbb{R},u)$ for $k > 1$ (resp. $ \frac{4}{3} > k > 1$) if $n \in \{3,4 \}$ (resp. $n=5$): this follows from the following proposition

\begin{prop}{\cite{triroyrad}}
Let $u$ be a solution of (\ref{Eqn:BarelySchrod}) with data
$u_0 \in \tilde{H}^{k^{'}}$, $k^{'} > \frac{n}{2}$. Assume that $u$ exists globally in time and that
$\| u \|_{L_{t}^{\frac{2(n+2)}{n-2}} L_{x}^{\frac{2(n+2)}{n-2}} (\mathbb{R})} < \infty$. Then
$Q_{k^{'}}(\mathbb{R},u) < \infty$.
\label{Prop:PersReg}
\end{prop}

\item Scattering: this part of the proof is contained in \cite{triroyrad,triroyjensen}. For sake of completeness we write here the
full details. It is enough to prove that $e^{- i t \triangle} u(t)$ has a limit as $t \rightarrow \infty$ in $\tilde{H}^{k}$. Let
$ 1 \gg \epsilon > 0$. There exists $A(\epsilon)$ large enough such that if $t_2 \geq t_1 \geq A(\epsilon)$ then
$\| u \|_{L_{t}^{\frac{2(n+2)}{n-2}} L_{x}^{\frac{2(n+2)}{n-2}} ([t_1,t_2])} \ll \epsilon$. We get from Plancherel theorem,
(\ref{Eqn:Strich}), (\ref{Eqn:NonlinEst}), and
Proposition \ref{Prop:EstHighReg}

\begin{equation}
\begin{array}{l}
\| e^{-i t_{1} \triangle} u(t_{1}) - e^{- i t_{2} \triangle} u(t_{2}) \|_{\tilde{H}^{k}} \\
=  \| e^{i(t_{2} - t_{1}) \triangle} u(t_{1}) - u(t_{2}) \|_{\tilde{H^{k}}} \\
\lesssim \left\| D \left( |u|^{\frac{4}{n-2}} u g(|u|) \right) \right\|_{L_{t}^{\frac{2(n+2)}{n+4}} L_{x}^{\frac{2(n+2)}{n+4}} ([t_1,t_2])}
+ \left\| D^{k} \left( |u|^{\frac{4}{n-2}} u g(|u|) \right) \right\|_{L_{t}^{\frac{2(n+2)}{n+4}} L_{x}^{\frac{2(n+2)}{n+4}} ([t_1,t_2])} \\
 \lesssim \| u \|^{\frac{4}{n-2}-}_{L_{t}^{\frac{2(n+2)}{n-2}} L_{x}^{\frac{2(n+2)}{n-2}} ([t_{1},t_{2}]) }  \\
\ll \epsilon \cdot
\end{array}
\label{Eqn:Diff}
\end{equation}
The Cauchy criterion is satisfied. Hence scattering.

\end{itemize}

\section{Proof of Proposition \ref{Prop:BoundLong}}
\label{Sec:ProofProp}

In this section we prove Proposition \ref{Prop:BoundLong}.  \\

\subsection{Objects: definition} \label{Subsec:Def}

In this subsection we define some objects that we use throughout the proof. \\
\vphantom{ab} \\
We first define some numbers and some functions that we will use whereever we use the Jensen inequality. \\
\\
Let $c_{\diamond}$ (resp. $C_{\diamond}$) be a fixed positive constant that is small (resp. large) enough so that all the statements in this section (and more generally in this paper) where $c_{\diamond}$ (resp. $C_{\diamond}$) appears are true. Let \footnote{The number $|k_{\diamond}|_{2}^{*} $ is such that
$\frac{1}{|k_{\diamond}|_{2}^{*}} = \frac{1}{2} - \frac{|k_{\diamond}|}{n} = \frac{1}{2} - \frac{k_{\diamond}}{n}$. Here we consider
$|k_{\diamond}|_{2}^{*} $ instead of $k_{\diamond_{2}}^{*}$  since the notation $k_{\diamond_{2}}^{*}$ might be confusing.}

\begin{equation}
\epsilon_{\diamond} := c_{\diamond} \Big( |k_{\diamond}|_{2}^{*} - 1_{2}^{*} \Big),
\nonumber
\end{equation}
with $ k_{\diamond}$ a constant such that $1 < k_{\diamond} < \frac{n}{2}$ and such that  $k_{\diamond} -1 \approx  k -1$ if $k-1 \ll 1$. \\
Let $h$ and $\tilde{g}$ be the following functions:

\begin{equation}
s \in \mathbb{R}^{+}: h(s) := g(s) -1, \; \tilde{g}(s):= g(s^{\frac{1}{2}}) \cdot
\nonumber
\end{equation}
Elementary considerations show that $\tilde{g} \left( s^{\epsilon_{\diamond}} \right) \geq \left( \frac{\epsilon_{\diamond}}{_{C_{\diamond}}} \right)^{\gamma} \tilde{g}(s)$. Hence $\tilde{h} (s^{2 \epsilon_{\diamond}}) \geq h(s)$ with

\begin{equation}
\tilde{h}(s):= \left( \frac{C_{\diamond}}{\epsilon_{\diamond}}  \right)^{\gamma} \tilde{g}(s) -1 \cdot
\nonumber
\end{equation}
Since $\gamma \ll 1$ (see Remark \ref{Rem:SmallGamma}), a study of the sign of the second derivative of $\tilde{g}$ and  $\tilde{h}$'s shows that $\tilde{g}$ and $\tilde{h}$ are concave on $(0, \infty)$ \footnote{We explain why we introduced $\tilde{g}$ and $\tilde{h}$. Observe that $g$ and $h$ are not concave on the whole interval $(0,\infty)$, as opposed to $\tilde{g}$ and $\tilde{h}$. So it is easier to apply the Jensen inequality to $\tilde{g}$ and $\tilde{h}$ than to $g$ and $h$. }. \\
Let $\theta_{\diamond} := \frac{ 1_{2}^{*} \Big( |k_{\diamond}|_{2}^{*} - (1_{2}^{*} + 2 \epsilon_{\diamond} ) \Big)}{_{(1_{2}^{*} +  2 \epsilon_{\diamond})
\Big( | k_{\diamond} |_{2}^{*} - 1_{2}^{*} \Big)}}$.
Observe that $ \frac{1}{1_{2}^{*} + 2 \epsilon_{\diamond}} = \frac{\theta_{\diamond}}{1_{2}^{*}} + \frac{1- \theta_{\diamond}}{_{ |k_{\diamond}|_{2}^{*}}} $. \\
\vphantom{ab} \\ \vphantom{ab}
We then define the functional $K_{cr}$ by

\begin{equation}
K_{cr}(f) := \| \nabla f \|^{2}_{L^{2}} - \| f \|^{1_{2}^{*}}_{L^{1_{2}^{*}}} .
\label{Eqn:DefK}
\end{equation}
Observe from (\ref{Eqn:W}) that

\begin{equation}
\begin{array}{ll}
K_{cr}(W) & = 0
\end{array}
\label{Eqn:ValKW}
\end{equation}

\subsection{A lemma}

We then prove a preliminary lemma.

\begin{lem}

There exists $\delta' \approx \delta^{\frac{1}{2}}$ such that for all $t \in J $


\begin{equation}
\begin{array}{l}
\int_{\mathbb{R}^{n}}  |\nabla u(t)|^{2} \; dx \leq (1 - \delta') \int_{\mathbb{R}^{n}} |\nabla W|^{2} \; dx \\
\\
K_{cr}(u(t)) \geq \delta' \int_{\mathbb{R}^{n}} |\nabla u(t)|^{2} \; dx,  \text{and}
\end{array}
\label{Eqn:MonotonVirial}
\end{equation}

\begin{equation}
\begin{array}{ll}
\| u(t) \|^{1_{2}^{*}}_{L^{1^{*}_{2}}} & \leq (1 - \delta') \| W \|^{1_{2}^{*}}_{L^{1_{2}^{*}}}
\end{array}
\label{Eqn:BoundPot}
\end{equation}

\label{lem:Monoton}
\end{lem}

\begin{proof}
Define

\begin{equation}
\begin{array}{ll}
\mathcal{F}:= \left\{ T \in J : \; (\ref{Eqn:MonotonVirial}) \; \text{holds for} \; t \in [0,T] \right\}
\end{array}
\nonumber
\end{equation}
We claim that $\mathcal{F} = J$. \\
\\
First observe that $0 \in \mathcal{F}$. Indeed using (\ref{Eqn:ValKW}) and $ E_{cr}(u_{0}) < ( 1 -  \delta ) E_{cr}(W) $
we get $ \| \nabla u_{0} \|_{L^{2}}^{2} < \| \nabla W \|^{2}_{L^{2}}$. But then this implies
(see \cite{kenmer}) that (\ref{Eqn:MonotonVirial}) holds at $t=0$. \\
Then we see that $ \mathcal{F}$ is closed by $\dot{H}^{1}-$ continuity of
$ t \rightarrow u(t)$, the embedding $ \dot{H}^{1} \hookrightarrow  L^{1_{2}^{*}} $, and elementary consequences of the Minkowski inequality. \\
It remains to prove that $\mathcal{F}$
is open. Let $T \in \mathcal{F}$. By continuity there exists $\beta > 0$ such that (\ref{Eqn:MonotonVirial}) holds for
$t' \in J \cap [0,T + \beta)$   with $\delta'$ replaced with $\frac{\delta'}{2}$. This implies that
$\| u(t') \|_{L^{1_{2}^{*}}} < \| W \|_{L^{1_{2}^{*}}}$. We claim the following: \\
\\
\underline{Claim}: We have

\begin{equation}
\begin{array}{l}
E_{cr}(u(t'))  <  \left( 1 - \frac{\delta}{4} \right)  E_{cr}(W).
\end{array}
\label{Eqn:BoundTildeE}
\end{equation}

\begin{proof}

An integration by parts shows that

\begin{equation}
\begin{array}{l}
P(|u_{0}(x)|) = \frac{g \left( |u_{0}(x)| \right) |u_{0}(x)|^{1_{2}^{*}}}{1_{2}^{*}}
- \frac{2 \gamma}{1_{2}^{*}} \int_{0}^{|u_{0}(x)|} \frac{s^{1_{2}^{*} + 1} \log^{\gamma -1} (2+ s^{2}) } {2+ s^{2}} \; ds \cdot
\end{array}
\nonumber
\end{equation}
Hence, from $ \gamma \ll \delta $ (see Remark \ref{Rem:SmallGamma}),
$ 1_{2}^{*} \int_{0}^{|u_{0}(x)|} s^{1_{2}^{*}-1} \; ds = |u_{0}(x)|^{1_{2}^{*}} $, and $E_{cr}(W)=
\left( \frac{1}{2} - \frac{1}{1_{2}^{*}} \right) \| W \|^{1_{2}^{*}}_{L^{1_{2}^{*}}}$ (see (\ref{Eqn:ValKW})), we see that
$E_{\gamma}(u_{0}) \leq \left( 1 - \frac{\delta}{2} \right) E_{cr}(W)$. By conservation of the energy we see that

\begin{equation}
\begin{array}{l}
E_{\gamma}(u(t')) \leq \left( 1  - \frac{\delta}{2} \right) E_{cr}(W) .
\end{array}
\nonumber
\end{equation}
Observe that $E_{cr}(f) = E_{\gamma}(f) + X (f) $ with

\begin{equation}
\begin{array}{l}
X(f) :=  \int_{\mathbb{R}^{n}}  \int_{(0,|f(x)|]} h(s) s^{1_{2}^{*} - 1} \; ds \; dx \cdot
\end{array}
\label{Eqn:Xf}
\end{equation}
Hence from the estimate above it is sufficient to prove that

\begin{equation}
\begin{array}{l}
X(u(t')) \leq \frac{\delta}{4} E_{cr}(W)\cdot
\end{array}
\label{Eqn:BoundXf}
\end{equation}
Let $\mu$ be the measure defined by $ d \mu := s^{1_{2}^{*} - 1} \mathbbm{1}_{(0,|u(t',x)|] } ds \; dx $. We have
$ X(u(t')) \lesssim \int_{\mathbb{R}^{n}} \int_{0}^{\infty} \tilde{h} \left( s^{2 \epsilon_{\diamond}}  \right) \; d \mu $.
We may assume WLOG that $\| u(t') \|^{1_{2}^{*}}_{L^{1_{2}^{*}}} \neq 0$. The Jensen inequality w.r.t the measure $\mu$  yields

\begin{equation}
\begin{array}{ll}
X (u(t')) \lesssim \| u(t') \|^{1_{2}^{*}}_{L^{1_{2}^{*}}} \tilde{h}
\left( \frac{\| u(t') \|^{1_{2}^{*} + 2 \epsilon_{\diamond}}_{L^{1_{2}^{*} + 2 \epsilon_{\diamond}}}}
{\| u(t') \|^{1_{2}^{*}}_{L^{1_{2}^{*}}}}  \right)
\end{array}
\nonumber
\end{equation}
Hence applying the interpolation inequality

\begin{equation}
\begin{array}{ll}
\| u(t') \|_{L^{1_{2}^{*} + 2 \epsilon_{\diamond}}} & \lesssim \| u(t') \|^{\theta_{\diamond}}_{L^{1_{2}^{*}}}
\| u(t') \|_{L^{|k_{\diamond}|_{2}^{*}}}^{1- \theta_{\diamond}}
\end{array}
\nonumber
\end{equation}
and the embedding $\tilde{H}^{k} \hookrightarrow L^{ |k_{\diamond}|_{2}^{*}}$  we see that

\begin{equation}
\begin{array}{l}
X(u(t')) \lesssim \| u(t')\|^{1_{2}^{*}}_{L^{1_{2}^{*}}}
\tilde{h} \left(  \frac{  \| u(t') \|_{L^{ |k_{\diamond}|_{2}^{*}}}^{ 2 c_{\diamond}  |k_{\diamond}|_{2}^{*}}}
{ \| u(t') \|_{L^{1_{2}^{*}}}^{2 c_{\diamond} 1_{2}^{*} } } \right) \lesssim \| u(t') \|^{1_{2}^{*}}_{L^{1_{2}^{*}}} \tilde{h} \left( \frac{M}
{ _{\| u(t') \|_{L^{1_{2}^{*}}}^{ 2 c_{\diamond} 1_{2}^{*} }}} \right)  \cdot
\end{array}
\label{Eqn:EstX2}
\end{equation}
Let $ Y :=  \frac{M} { _{\| u(t') \|_{L^{1_{2}^{*}}}^{ 2 c_{\diamond} 1_{2}^{*} }}} $. Observe that there exists a constant $ C > 0 $ such that
$ Y \geq C M $. A computation of the derivative of $ x \rightarrow x^{-\frac{1}{2 c_{\diamond}}} \tilde{h}(x) $ shows that this
function decreases for $ x \geq C M $, taking into account that $\gamma \ll 1 $ (see Remark \ref{Rem:SmallGamma}). Hence

\begin{equation}
\begin{array}{l}
X (u(t')) \lesssim  \left( \frac{M}{Y} \right)^{\frac{1}{2 c_{\diamond}}}  \tilde{h}(Y)  \lesssim \tilde{h} (C M)  \ll \delta \cdot
\end{array}
\nonumber
\end{equation}
Hence (\ref{Eqn:BoundXf}) holds, taking into account that $E_{cr}(W) = \left( \frac{1}{2} - \frac{1}{1_{2}^{*}} \right) \| W \|^{1_{2}^{*}}_{L^{1_{2}^{*}}}$.

\end{proof}

From (\ref{Eqn:BoundTildeE}) and $\| \nabla u(t') \|^{2}_{L^{2}} < \| \nabla W \|^{2}_{L^{2}}$  we see that
(\ref{Eqn:MonotonVirial}) holds for $t=t'$: see \cite{kenmer}. Hence $t' \in \mathcal{F}$ and $\mathcal{F}$ is open.
Hence $ \mathcal{F} = J $,  which implies that (\ref{Eqn:BoundPot}) holds.

\end{proof}

\subsection{The proof}

We prove now Proposition \ref{Prop:BoundLong} by using this lemma and concentration techniques (see e.g \cite{bourg,taorad}). \\
We divide the interval $J$ into subintervals $(J_{l}:=[\bar{t}_{l},\bar{t}_{l+1}])_{1  \leq  l \leq L}$ such that

\begin{equation}
\begin{array}{ll}
\| u \|^{\frac{2(n+2)}{n-2}}_{L_{t}^{\frac{2(n+2)}{n-2}}  L_{x}^{\frac{2(n+2)}{n-2}} (J_{l}) } & = \eta_{1}, \; \text{and}
\end{array}
\label{Eqn:Conc}
\end{equation}

\begin{equation}
\begin{array}{ll}
\| u \|^{\frac{2(n+2)}{n-2}}_{L_{t}^{\frac{2(n+2)}{n-2}}  L_{x}^{\frac{2(n+2)}{n-2}} (J_{L}) } & \leq \eta_{1},
\end{array}
\label{Eqn:LastConc}
\end{equation}
with $ 0 < \eta_{1} \ll 1 $. In view of (\ref{Eqn:BoundLong}), we may replace WLOG the $``\leq''$ sign with the $``=''$ sign in (\ref{Eqn:LastConc}).  \\
Recall the notion of exceptional intervals and the notion of unexceptional intervals (such a notion appeared
in the study of (\ref{Eqn:EnergyCrit}) in \cite{taorad}). Let $ \eta_{1} \gg \eta_{2} > 0 $ be a positive constant that is small enough and such that all the estimates and statements below are true. An interval $J_{l_{0}} = [\bar{t}_{l_{0}}, \bar{t}_{l_{0}+1}]$ of the partition $(J_{l})_{ 1 \leq l \leq L}$ is exceptional if

\begin{equation}
\begin{array}{ll}
\| u_{l,0} \|^{\frac{2(n+2)}{n-2}}_{L_{t}^{\frac{2(n+2)}{n-2}} L_{x}^{\frac{2(n+2)}{n-2}} (J_{l_{0}})  } + \| u_{l,a}
\|^{\frac{2(n+2)}{n-2}}_{L_{t}^{\frac{2(n+2)}{n-2}} L_{x}^{\frac{2(n+2)}{n-2}} (J_{l_{0}}) } & \geq \eta_{2}  \cdot
\end{array}
\end{equation}
In view of the embedding $ \| f \|_{L^{\frac{2(n+2)}{n-2}}} \lesssim  \| D f \|_{L^{\frac{2n(n+2)}{n^{2}+4}}} $, (\ref{Eqn:Strich}), and Lemma \ref{lem:Monoton} we have

\begin{equation}
\begin{array}{ll}
\card{\{ J_{l}: \, J_{l} \, \mathrm{exceptional} \}} & \lesssim \eta_{2}^{-1} \cdot
\end{array}
\label{Eqn:BoundCardExcep}
\end{equation}
In the sequel we use the estimate $g(M) \lesssim 1$ that follows from Remark \ref{Rem:SmallGamma}. By applying (\ref{Eqn:Strich}), the estimate
$k' > \frac{n}{2}: \; Q_{\frac{n}{2}+} (J,u) \lesssim Q_{k'}(J,u) $ that follows from interpolation, Lemma \ref{lem:Monoton}, and (\ref{Eqn:NonlinEst}) with $j=1$ to $J:=J_l$ for $J_l$ unexceptional interval, we may use the arguments of \cite{taorad} to get two results
\footnote{This is at this stage that we have to assume that $\eta_{2} \ll \eta_{1}$: see \cite{taorad}.}. The first result says there is a ball for which we have a mass concentration:

\begin{res}
There exist an $x_{l} \in \mathbb{R}^{n}$, two constants $ 0 <  c' \ll  1 $ and $C' \gg  1$ such that for each unexceptional interval $J_{l}$ and
for $t \in J_{l}$

\begin{equation}
\begin{array}{ll}
Mass \left( u(t), B(x_{l}, C'  |J_{l}|^{\frac{1}{2}} )   \right) & \geq c' |J_{l}|^{\frac{1}{2}} \cdot
\end{array}
\label{Eqn:ConcentrationMass}
\end{equation}

\label{res:ConcentrationMass}
\end{res}
The second result shows that in fact there is a mass concentration around the origin (the proof uses the radial symmetry)

\begin{res}
There exist a positive constant $\ll  1$ (that we still denote by $c'$ ) and a constant $\gg 1$ (that we still denote by $C'$ ) such that on each unexceptional interval $J_{l}$ we have

\begin{equation}
\begin{array}{ll}
Mass \left( u(t),  B(0, C' |J_{l}|^{\frac{1}{2}}  )  \right) & \geq c' |J_{l}|^{\frac{1}{2}} \cdot
\end{array}
\label{Eqn:ConcMassZero}
\end{equation}

\label{res:ConcMassZero}
\end{res}

Let $\tilde{J} := J_{i_0} \cup ... \cup J_{i_1}$ be a sequence of consecutive unexceptional intervals. Let $L_{\tilde{J}}$ be the
number of unexceptional intervals making $\tilde{J}$. Observe that

\begin{equation}
\begin{array}{l}
Number \; of \; \tilde{J}s \; \lesssim \eta_{2}^{-1} \cdot
\end{array}
\label{Eqn:CardSeq}
\end{equation}
We claim that one of the intervals $J_{l} \in \tilde{J}$ is large. More precisely

\begin{res}
There exists a positive constant $\ll  1$ (that we still denote by $c'$) and $\tilde{l} \in [i_0,..,i_1]$ such that

\begin{equation}
\begin{array}{ll}
|J_{\tilde{l}}|  &  \geq (c')^{ \delta^{- \frac{1}{2}}}  |\tilde{J}| \cdot
\end{array}
\label{Eqn:DistribInterv}
\end{equation}

\label{res:DistribInterv}
\end{res}

\begin{proof}
Let $a$ be a smooth function. Let $t \in \tilde{J}$. Let $v$ be a solution of $i \partial_{t} v + \triangle v =G$.
Let $\{ G, f \}_{p} := \Re( G \overline{\nabla f} - f \overline{\nabla G} )$. Recall the following facts
(see, e.g., \cite{collkeelstafftaktao}):

\begin{itemize}

\item if $G$ is of the form $ G(z) := F'(|z|^{2}) z $ then $ \{ G, f \}_{p} = - \nabla H(|f|^{2}) $ with
$H(x) := x F^{'}(x) - F(x) $.

\item We have

\begin{equation}
\begin{array}{ll}
\partial_{t} M_a  = \int (- \triangle \triangle a) |v|^{2} + 4 \partial^{2}_{x_j x_{k}} a
\Re (\overline{\partial_{x_j} v} \partial_{x_k} v )  + 2 \partial_{x_j} a \{ G, v \}_{p}^{j} \; dx,
\end{array}
\label{Eqn:DerMa}
\end{equation}
with $M_a(t):= \int 2 \partial_{x_j} a \Im(\bar{v} \partial_{x_{j}}v) \; dx $.

\end{itemize}
Let $m \in \mathbb{R}^{+}$ and $G(z) := - |z|^{\frac{4}{n-2}} z g(|z|) $. Let
$a(x) := m^{2} \phi \left( \frac{|x|}{m} \right) $ with $\phi$ a smooth, radial, and compactly supported function such that
$\phi(x)=|x|^{2}$ for $|x| \leq 1$. Then we have the well-known virial identity
\footnote{We define $\langle f,g \rangle := \Re  \left( \int_{\mathbb{R}^{n}} f(x) \bar{g}(x) \; dx \right) $.}

\begin{equation}
\begin{array}{l}
2m \; \partial_{t}  \left\langle  \phi^{'} \left( \frac{|x|}{m} \right) \frac{x_j}{|x|} v, - i \,  \partial_{x_j} v \right\rangle  = \int (- \triangle \triangle a) \, |v|^{2} \, + \, 4 \partial^{2}_{x_j x_{k}} a \,
\Re \left( \overline{\partial_{x_j} v} \partial_{x_k} v \right) \,  + \, 2 \triangle a  H(|v|^{2})  \; dx,
\end{array}
\label{Eqn:VirialPrim}
\end{equation}
with

\begin{equation}
\begin{array}{l}
H(y) := -y^{\frac{n}{n-2}} \tilde{g}(y)  + \int_{0}^{y} s^{\frac{2}{n-2}} \tilde{g}(s) \; ds \cdot
\end{array}
\nonumber
\end{equation}
Observe from Remark \ref{Rem:SmallGamma} that $\gamma \ll 1$. Hence, integrating by parts once the second term of $H(y)$ we see that

\begin{equation}
\begin{array}{l}
H (y)  = \left( \frac{n-2}{n} -1 \right) y^{\frac{n}{n-2}} \tilde{g}(y)
- \frac{n-2}{n} \int_{0}^{y} s^{\frac{n}{n-2}} \tilde{g}^{'}(s) \; ds
, \; \text{which implies that} \;  H(y) \approx - y^{\frac{n}{n-2}} \tilde{g}(y) \cdot
\end{array}
\label{Eqn:FormH}
\end{equation}
Hence one can write \footnote{Assume first that $n \in \{ 3,4 \}$. The above argument shows that
(\ref{Eqn:Virial}) holds for smooth solutions (i.e $\tilde{H}^{p}-$ solutions with $p$ large enough); in order to prove
(\ref{Eqn:Virial}) for $\tilde{H}^{k}-$ solutions, $k \in I_{n}$, one uses a standard approximation argument of $\tilde{H}^{k}-$
solutions with smooth solutions. If $n=5$ then proceed similarly as in Footnote \ref{Foot:EnergyRem} to obtain a identity similar to (\ref{Eqn:Virial}) for smooth solutions of the ``smoothed'' equation; then one should take limit in  $\tilde{H}^{k}$ by a standard approximation with smooth solutions.
}

\begin{equation}
\begin{array}{ll}
2m \, \partial_{t} \left \langle \phi^{'} \left( \frac{|x|}{m} \right) \frac{x_j}{|x|} u,  -i \,  \partial_{x_j} u \right \rangle  &  =
8 \int_{|x| \leq m} |\nabla u(t)|^{2} - |u(t)|^{1_{2}^{*}} \, dx - 8 \int_{|x| \leq m} h(|u(t)|) |u(t)|^{1_{2}^{*}} \; dx  \\
& \\
& + O \left( X_{m}(t) \right) + O \left( Y_{m}(t) \right) \\
& \\
& \geq 8 K_{cr}( \chi_m u(t) ) -  8 \int_{|x| \leq m} h(|u(t)|) |u(t)|^{1_{2}^{*}} \; dx \\
& \\
& + O \left( X_{m}(t) \right) + O \left( Y_{m}(t) \right) \cdot
\end{array}
\label{Eqn:Virial}
\end{equation}
Here $\chi_{m}(x) := \chi \left( \frac{|x|}{m} \right)$ with $\chi$ a
smooth function compactly supported on $B(0,2)$ such that $\chi(x)=1$ if $|x| \leq 1$,
$X_{m}(t) := \int_{|x| \leq 2m} \int_{0}^{|u(t)|^{2}} t'^{\frac{n}{n-2}} \tilde{g}^{'}(t') \; dt' \; dx $, and
$Y_{m}(t) := \int_{|x| \geq m} \frac{m}{|x|} \left( |\partial_{r} u(t)|^{2} +  \frac{|u(t)|^{2}}{|x|^{2}} + |u(t)|^{1_{2}^{*}} g(|u(t)|) \right) \; d x$. \\
Observe also from Remark \ref{Rem:SmallGamma} that $\gamma \ll \delta^{\frac{1}{2}}$. Hence, using also (\ref{Eqn:BoundPot})

\begin{equation}
\begin{array}{l}
X_{m}(t) \ll  \delta^{\frac{1}{2}}  \left( \ \int_{|x| \leq m} |u(t)|^{1_{2}^{*}} \; dx + 1 \right)
\end{array}
\nonumber
\end{equation}

\underline{Claim}: there exists a constant $\tilde{c}_{1} \approx \delta^{\frac{1}{2}}$ such that

\begin{equation}
\begin{array}{l}
K_{cr} \left( \chi_m u (t) \right) \geq \tilde{c}_{1} \| \chi_{m} u(t) \|^{1_{2}^{*}}_{L^{1_{2}^{*}}}  \cdot
\end{array}
\label{Eqn:KMon}
\end{equation}

\begin{proof}

Indeed (see \cite{roynak} for a similar argument, see also \cite{nakschlagkg} for the subcritical nonlinearities) we see from
(\ref{Eqn:W}) and (\ref{Eqn:BoundPot}) that

\begin{equation}
\begin{array}{l}
\left\| \chi_{m} u(t) \right\|^{1_{2}^{*}}_{L^{1_{2}^{*}}}  < \left\| u(t) \right\|^{1_{2}^{*}}_{L^{1_{2}^{*}}} <  (1 -\delta') \| W \|^{1_{2}^{*}}_{L^{1_{2}^{*}}};
\end{array}
\nonumber
\end{equation}
Hence we see from the sharp Sobolev inequality that there exists $ \tilde{c}_{1} \approx \delta^{\frac{1}{2}}$ such that

\begin{equation}
\begin{array}{ll}
K_{cr} \left( \chi_m u(t) \right) & \geq \left\| \nabla \left( \chi_m u(t) \right) \right\|_{L^{2}}^{2} \left( 1 - C_{*}^{2}
\left\| \chi_m u(t) \right\|^{1_{2}^{*}-2}_{L^{1_{2}^{*}}}  \right) \\
& \geq \left\| \nabla \left( \chi_m u (t) \right) \right\|_{L^{2}}^{2} \left( 1 - C_{*}^{2} (1 -\delta')^{\frac{1_{2}^{*}-2}{1_{2}^{*}}}
\| W \|^{1_{2}^{*} - 2}_{L^{1_{2}^{*}}} \right) \\
& \geq \tilde{c}_{1} \left\| \nabla \left( \chi_m u (t) \right) \right\|_{L^{2}}^{2} \cdot
\end{array}
\nonumber
\end{equation}
Hence (\ref{Eqn:KMon}) holds.

\end{proof}

Let $\bar{K} \in \mathbb{N}^{*}$ and let $\bar{m} := 2^{- \bar{K}} m$. Define $ \sum \limits_{m'= \bar{m}}^{m} a_{m'} := \sum \limits_{k=0}^{\bar{K}} a_{2^{-k} m} $.
The following claim is used in the sequel:\\
\\
\underline{Claim}: \\
\\
Let $Z(t):= \sum \limits_{m'= \bar{m}}^{m} \int_{|x| \geq m'} \frac{m'}{|x|} |u(t)|^{1_{2}^{*}} g(|u(t)|) \; dx$. Then
$ Z(t) \lesssim  1 $ .

\begin{proof}

Clearly $Z (t) \lesssim \left( \frac{C_{\diamond}}{\epsilon_{\diamond}} \right)^{\gamma}
\sum \limits_{m'=\bar{m}}^{m}  \int_{  } \frac{m'}{|x|} \tilde{g}(|u(t)|^{2 \epsilon_{\diamond}}) |u(t)|^{1_{2}^{*}}
\mathbbm{1}_{|x| \geq m'} \; dx \cdot $. The Fubini theorem shows that $ \sum \limits_{m' = \bar{m}}^{m} \int_{|x| \geq m'} \frac{m'}{|x|} |u(t)|^{^{|k_{\diamond}|_{2}^{*} }} \; dx \lesssim
\| u(t) \|^{^{|k_{\diamond} |_{2}^{*}}}_{L^{ | k_{\diamond} |_{2}^{*}}} $. Let $\mu(t)$ be the measure defined by $ d \mu(t) := \sum \limits_{m'=\bar{m}}^{m} \frac{m^{'}}{|x|} |u(t)|^{1_{2}^{*}} \; \mathbbm{1}_{|x| \geq m'} dx $. Let $\bar{\mu}(t)$ be the measure defined by $ d \bar{\mu}(t) := \sum \limits_{m'=\bar{m}}^{m} \frac{m'}{|x|} \mathbbm{1}_{|x| \geq m'} dx $.
We have

\begin{equation}
\begin{array}{ll}
Z(t) & \lesssim  \left( \frac{C_{\diamond}}{\epsilon_{\diamond}} \right)^{\gamma} \int_{\mathbb{R}^{n}}
\tilde{g} \left( |u(t)|^{2 \epsilon_{\diamond}} \right)  \; d \mu(t) \cdot
\end{array}
\nonumber
\end{equation}
We may assume WLOG that $ \| u(t) \|^{1_{2}^{*}}_{L^{1_{2}^{*}}(d \bar{\mu}(t))} \neq 0 $. We get from the Jensen inequality w.r.t the measure $\mu(t)$

\begin{equation}
\begin{array}{ll}
Z(t) & \lesssim \left( \frac{C_{\diamond}}{ \epsilon_{\diamond}} \right)^{\gamma} \| u(t) \|^{1_{2}^{*}}_{L^{1_{2}^{*}}(d \bar{\mu}(t))}
\tilde{g} \left( \frac{\| u(t) \|^{1_{2}^{*} + 2 \epsilon_{\diamond}}_{L^{1_{2}^{*} + 2 \epsilon_{\diamond}}(d \bar{\mu}(t) )  }}
{ \| u(t) \|^{1_{2}^{*}}_{L^{1_{2}^{*}}(d \bar{\mu}(t)) } } \right)  \\
& \lesssim \left( \frac{C_{\diamond}}{\epsilon_{\diamond}} \right)^{\gamma} \tilde{g}
\left( \frac{\| u(t)\|^{2 c_{\diamond} |k_{\diamond}|_{2}^{*}}_{L^{^{| k_{\diamond}|_{2}^{*}}}(d \bar{\mu}(t))} }{ \| u(t) \|^{2 c_{\diamond} 1_{2}^{*}}_{L^{1_{2}^{*}}(d \bar{\mu}(t))}  }
\right) \| u(t) \|^{1_{2}^{*}}_{L^{1_{2}^{*}}(d \bar{\mu}(t))}  \\
& \lesssim \left( \frac{C_{\diamond}}{\epsilon_{\diamond}} \right)^{\gamma} \tilde{g}
\left( \frac{\| u(t)\|^{2 c_{\diamond} |k_{\diamond}|_{2}^{*}}_{\tilde{H}^{k}}} { \| u(t) \|^{2 c_{\diamond} 1_{2}^{*}}_{L^{1_{2}^{*}}(d \bar{\mu}(t))}  }
\right) \| u(t) \|^{1_{2}^{*}}_{L^{1_{2}^{*}}(d \bar{\mu}(t))} \\
& \lesssim \left( \frac{C_{\diamond}}{\epsilon_{\diamond}} \right)^{\gamma}
\tilde{g} \left( \frac{M}{ \| u(t) \|^{2 c_{\diamond} 1_{2}^{*}}_{L^{1_{2}^{*}}(d \bar{\mu}(t))}} \right)  \| u(t) \|^{1_{2}^{*}}_{L^{1_{2}^{*}}(d \bar{\mu}(t))}
\end{array}
\nonumber
\end{equation}
where at the second line we apply the interpolation inequality

\begin{equation}
\begin{array}{l}
\| u(t) \|_{L^{1_{2}^{*} + 2 \epsilon_{\diamond}}(d \bar{\mu}(t))} \lesssim
\| u(t) \|^{\theta_{\diamond}}_{L^{1_{2}^{*}}(d \bar{\mu}(t))}
\| u(t) \|^{1 - \theta_{\diamond}}_{ L^{^{|k_{\diamond}|_{2}^{*}}}(d \bar{\mu}(t))}
\end{array}
\label{Eqn:HolderMeas}
\end{equation}
and at the fourth line we used the embedding $ \tilde{H}^{k} \hookrightarrow L^{^{|k_{\diamond}|_{2}^{*}}}(d \bar{\mu}(t))$. Let
$ Y:= \frac{M} {\| u(t) \|^{2 c_{\diamond} 1_{2}^{*}}_{L^{1_{2}^{*}}(d \bar{\mu}(t))} } $. Proceeding similarly as below
(\ref{Eqn:EstX2}) and taking into account Lemma \ref{lem:Monoton} we see that there exists a constant $C > 0$ such that

\begin{equation}
\begin{array}{l}
Z(t) \lesssim  \left( \frac{C_{\diamond}}{\epsilon_{\diamond}} \right)^{\gamma}  \left( \frac{M}{Y} \right)^{\frac{1}{2 c_{\diamond}}}
\tilde{g}(Y)  \lesssim \left( \frac{C_{\diamond}}{\epsilon_{\diamond}} \right)^{\gamma}  \tilde{g} (C M) \lesssim 1 \cdot
\end{array}
\nonumber
\end{equation}
Hence $Z(t) \lesssim 1$.

\end{proof}
We continue the proof of Result \ref{res:DistribInterv}. The Hardy inequality, the Fubini theorem, and Lemma \ref{lem:Monoton} show that

\begin{equation}
\begin{array}{ll}
\sum \limits_{m'= \bar{m}}^{m} \int_{|x| \geq m'} \frac{m'}{|x|} \left( |\partial_{r} u(t)|^{2} + \frac{|u(t)|^{2}}{|x|^{2}} \right) \; dx & \lesssim 1 \cdot
\end{array}
\nonumber
\end{equation}
In the sequel we let $m:= |\tilde{J}|^{\frac{1}{2}}$. \\
\\
\underline{Claim}: Let $\tilde{c}_2 > 0 $. Assume that $\tilde{h}(\tilde{c}_2^{-1} M) \ll \delta$. One can find $\tilde{C}_2 \gg 1$ large enough such that there exists $\tilde{t} \in \tilde{J}$ such that
s
\begin{equation}
\begin{array}{l}
\int_{|x| \leq \bar{m}} |u(\tilde{t})|^{1_{2}^{*}} \; dx \leq \tilde{c}_2,
\end{array}
\label{Eqn:DecayPotOneTime}
\end{equation}
with $\bar{K} := \tilde{C}_2  \delta^{-\frac{1}{2}}$.\\

\begin{proof}
Let $\tilde{C}_{2} \gg 1$ be a constant large enough such all the estimates and the statements below are true. Assume that
(\ref{Eqn:DecayPotOneTime}) does not hold. Hence for all $t \in J$ and for all $\bar{m} \leq m' \leq m$ we have

\begin{equation}
\begin{array}{l}
\int_{|x| \leq m'} |u(t)|^{1_{2}^{*}} \; dx \geq \tilde{c}_{2} \cdot
\end{array}
\label{Eqn:Contr}
\end{equation}
Let $ Z(t) := \sum \limits_{m^{'} = \bar{m}}^{m} \int_{|x| \leq m'} h(|u(t)|) |u(t)|^{1_{2}^{*}} \; dx $. Let $\mu_{m'}(t)$ be the measure defined by $d \mu_{m'}(t) :=  |u(t)|^{1_{2}^{*}} \; \mathbbm{1}_{|x| \leq m'} \; dx  $. Let $\bar{\mu}_{m'}(t)$ be the measure defined by $d \bar{\mu}_{m'}(t) := \mathbbm{1}_{|x| \leq m'} \; dx $. We get from the Jensen inequality w.r.t $ \mu_{m'}(t) $ and from (\ref{Eqn:Contr})

\begin{equation}
\begin{array}{ll}
Z(t) &  \lesssim \sum  \limits_{m' = \bar{m}}^{m} \| u(t) \|^{1_{2}^{*}}_{L^{1_{2}^{*}}(d \bar{\mu}_{m'}(t))}
 \tilde{h} \left( \frac{ \| u(t) \|^{1_{2}^{*} + 2 \epsilon_{\diamond}}_{L^{1_{2}^{*} + 2 \epsilon_{\diamond}} (d \bar{\mu}_{m'}(t)) }}
 { \| u(t') \|^{1_{2}^{*}}_{L^{1_{2}^{*}} (d \bar{\mu}_{m'}(t)) } } \right) \\
& \\
& \lesssim  \sum \limits_{m' = \bar{m}}^{m} \| u(t) \|^{1_{2}^{*}}_{L^{1_{2}^{*}}(d \bar{\mu}_{m'}(t))}
\tilde{h}
\left(
\frac{ \| u(t) \|_{L^{^{|k_{\diamond}|_{2}^{*}}} (d \bar{\mu}_{m'}(t)) }^{2  c_{\diamond} |k_{\diamond}|_{2}^{*}}}
{ { \| u(t) \|^{2  c_{\diamond} 1_{2}^{*}}_{L^{1_{2}^{*}}(d \bar{\mu}_{m'}(t))}}  }
\right) \\
&  \\
& \lesssim    \sum \limits_{m' = \bar{m}}^{m} \| u(t) \|^{1_{2}^{*}}_{L^{1_{2}^{*}}(d \bar{\mu}_{m'}(t))}
\tilde{h} \left( \frac{M}{ \| u(t) \|^{2 c_{\diamond} 1_{2}^{*} }_{L^{1_{2}^{*}}( d \bar{\mu}_{m'}(t)) } } \right) \\
& \ll \delta \sum \limits_{m' = \bar{m}}^{m} \int_{|x| \leq m'}  |u(t)|^{1_{2}^{*}} \; dx,
\end{array}
\nonumber
\end{equation}
where at the second line we use (\ref{Eqn:HolderMeas}), replacing $\bar{\mu}(t) $ with $ \bar{\mu}_{m'}(t) $. \\
The Cauchy-Schwarz inequality, the fundamental theorem of calculus, Lemma \ref{lem:Monoton} and the Hardy inequality imply that

\begin{equation}
\begin{array}{ll}
\left| \int_{\tilde{J}} 2 m \; \partial_{t}  \langle  \phi^{'} \left( \frac{|x|}{m} \right) \frac{x_{j}}{|x|} u, - i \partial_{x_{j}} u  \rangle  \; dt \right|
\lesssim m^{2} \; \sup_{ t \in \tilde{J}} \left( \left\| \frac{u(t)}{x} \right\|_{L^{2}} \| \nabla u(t) \|_{L^{2}} \right)  \lesssim m^{2} \cdot
\end{array}
\nonumber
\end{equation}
Hence, in view of the above estimates, Remark \ref{Rem:SmallGamma}, and (\ref{Eqn:Virial}) we have, using an argument in \cite{taofoc}

\begin{equation}
\begin{array}{ll}
\sum \limits_{m' = \bar{m}}^{m} \int_{\tilde{J}} \int_{|x| \leq m'} |u(t)|^{1_{2}^{*}} \; dx \; dt & \lesssim \delta^{-\frac{1}{2}}  m^{2} ,
\end{array}
\nonumber
\end{equation}
from which we get

\begin{equation}
\begin{array}{ll}
\int_{\tilde{J}} \int_{|x| \leq \bar{m}} |u(t)|^{1_{2}^{*}} \; dx \; dt & \lesssim \delta^{-\frac{1}{2}} \bar{K}^{-1} |\tilde{J}| \cdot
\end{array}
\nonumber
\end{equation}
Hence there exists $\bar{t} \in \tilde{J}$ such that

\begin{equation}
\begin{array}{ll}
\int_{|x| \leq  \bar{m}} |u(\bar{t})|^{1_{2}^{*}} \; dx & \lesssim \delta^{-\frac{1}{2}}  \bar{K}^{-1}  \cdot
\end{array}
\label{Eqn:Decay}
\end{equation}
This is a contradiction.
\end{proof}
But then, by choosing appropriately $\tilde{c}_{2}$,
we see that (\ref{Eqn:DistribInterv}) holds: if not this would violate (\ref{Eqn:DecayPotOneTime})
from (\ref{Eqn:ConcMassZero}) and H\"older inequality.

\end{proof}

Next we recall a crucial algorithm due to Bourgain \cite{bourg} to prove that there are many of those intervals that concentrate.

\begin{res}
Let $ \eta:= (c')^{ \delta^{-\frac{1}{2}}}$. There exist a time $\bar{t}$, $K > 0$ and intervals
$J_{i_{l_{1}}}$, ...., $J_{i_{l_{K}}}$ such that

\begin{equation}
\begin{array}{ll}
|J_{i_{l_{1}}}| \geq  2 |J_{i_{l_{2}}}| ... \geq 2^{k-1} |J_{i_{l_{k}}}| ... \geq 2^{K-1} |J_{i_{l_{K}}}|,
\end{array}
\label{Eqn:SizeJ}
\end{equation}

\begin{equation}
\begin{array}{ll}
dist(\bar{t}, J_{i_{l_{k}}}  ) \leq \eta^{-1} |J_{i_{l_k}}|,
\end{array}
\label{Eqn:ControlDist}
\end{equation}
and

\begin{equation}
\begin{array}{ll}
K & \geq - \frac{\log{(L_{\tilde{J}})}}{2 \log { \left(  \frac{\eta}{8} \right)}  } \cdot
\end{array}
\label{Eqn:LowerBdK}
\end{equation}

\end{res}
A proof of this result in such a state can be found in \cite{triroyrad} (see also \cite{taorad}
from which the proof is inspired). \\
\\
With this result in mind we prove that $L_{\tilde{J}} < \infty$. More  precisely

\begin{res}
There exists one constant  $C_{1} \gg  1$  such that

\begin{equation}
\begin{array}{l}
L_{\tilde{J}} \leq  C_{1}^{C_{1}^{\delta^{-\frac{1}{2}} }}
\end{array}
\label{Eqn:BoundLFin}
\end{equation}

\label{Res:BoundL}
\end{res}

\begin{proof}
Here we use an argument in \cite{taorad}. \\
Let $C \gg 1$ be a constant that is allowed to change 
from one line to the other one and such that all the estimates below are true. Let
$R_{i_{l_k}} := C^{\delta^{-\frac{1}{2} } }
|J_{i_{l_k}}|^{\frac{1}{2}}$. By  Result \ref{res:ConcentrationMass} and by (\ref{Eqn:UpBdDerivM})  we have

\begin{equation}
\begin{array}{ll}
Mass \left( u(t),  B(x_{i_{l_{k}}},R_{i_{l_k}}) \right) & \geq c'  |J_{i_{l_k}}|^{\frac{1}{2}}
\end{array}
\label{Eqn:LowerBondMass2}
\end{equation}
for all $t \in J_{i_{l_{k}}}$. By (\ref{Eqn:UpBdDerivM}) and (\ref{Eqn:ControlDist}) we see that (\ref{Eqn:LowerBondMass2}) holds for
$t=\bar{t}$ with $c'$ substituted with $\frac{c'}{2}$. On the other hand we see that by (\ref{Eqn:MassControl}) and
(\ref{Eqn:SizeJ}) that \footnote{Notation: $\sum\limits_{k^{'}=k+N}^{K} a_{k^{'}}=0$, if $k^{'}> K$ }

\begin{equation}
\begin{array}{ll}
\sum\limits_{k^{'}=k+N}^{K} \int_{B(x_{i_{l_{k'}}} ,R_{i_{l_{k'}}} )} |u(\bar{t})|^{2} \, dx & \lesssim  \left( \frac{1}{2^{N}} + \frac{1}{2^{N+1}}.... +
\frac{1}{2^{K-k}} \right)  R_{i_{l_k}}^{2} \\
& \lesssim \frac{1}{2^{N-1}}  R_{i_{l_k}}^{2} \cdot
\end{array}
\nonumber
\end{equation}
Now we let  $N := C \delta^{-\frac{1}{2}}  $  so that $ \frac{R_{i_{l_k}}^{2}}{2^{N-1}} \ll \frac{1}{8} (c')^{2}
|J_{i_{l_k}}| $. By (\ref{Eqn:LowerBondMass2}) we have

\begin{equation}
\begin{array}{ll}
\sum\limits_{k^{'}=k+N}^{K} \int_{B(x_{i_{l_{k^{'}}}} ,R_{i_{l_{k'}}} )} |u(\bar{t})|^{2} \, dx & \leq
\frac{1}{2}  \int_{B(x_{i_{l_{k}}},R_{i_{l_k}})} |u(\bar{t})|^{2} \, dx \cdot
\end{array}
\nonumber
\end{equation}
Therefore

\begin{equation}
\begin{array}{ll}
\int_{_{B(x_{i_{l_{k}}}, R_{i_{l_k}}) / \bigcup_{k^{'}=k+N}^{K} B (x_{i_{l_{k^{'}}}} ,R_{i_{l_{k'}}})  } } \   |u(\bar{t})|^{2} \, dx & \geq \frac{1}{2}  \int_{B(x_{i_{l_{k}}},R_{i_{l_k}})} |u(\bar{t})|^{2} \, dx \\
& \geq \frac{(c')^{2}}{8} |J_{i_{l_k}}|,
\end{array}
\nonumber
\end{equation}
and by H\"older inequality, there exists a small constant (that we still denote by $c'$) such that

\begin{equation}
\begin{array}{ll}
\int_{_{B(x_{i_{l_{k}}}, R_{i_{l_k}}) / \bigcup_{k^{'}=k+N}^{K} B (x_{i_{l_{k^{'}}}} ,R_{i_{l_{k'}}})  } } \   |u(\bar{t})|^{1_{2}^{*}} \, dx & \geq
(c')^{ \delta^{-\frac{1}{2}}},
\end{array}
\nonumber
\end{equation}
and after summation over $k$, we get

\begin{equation}
\begin{array}{l}
\frac{K}{N} (c')^{ \delta^{-\frac{1}{2}} } \lesssim 1,
\end{array}
\nonumber
\end{equation}
from $\sum \limits_{k=1}^{K} \chi_{ _{ B(x_{i_{l_{k}}}, R_{i_{l_k}}) / \cup_{k^{'}=k+N}^{K}
B (x_{i_{l_{k^{'}}}} ,R_{i_{l_{k'}}})}} \leq N$  and (\ref{Eqn:BoundPot}). Hence from (\ref{Eqn:LowerBdK}) we see that there exists a constant $C_{1} \gg  1$ such that (\ref{Eqn:BoundLFin}) holds. \\
\end{proof}

In view of Result \ref{Res:BoundL}, (\ref{Eqn:Conc}), (\ref{Eqn:BoundCardExcep}) and (\ref{Eqn:CardSeq}), we see that
(\ref{Eqn:BoundLong}) holds. \\

\section{Appendixes}

\subsection{Appendix $A$: equivalence of assumptions in  \cite{kenmer}}

Let $\delta > 0$. We denote by $Ass \; 0$, $Ass \; 1$, and $Ass \; 2$ the following assumptions

\begin{equation}
\begin{array}{ll}
Ass \; 0: & E_{cr}(u_{0}) < (1- \delta) E_{cr}(W) \\
Ass \; 1: & \| \nabla u_{0} \|_{L^{2}} < \| \nabla W \|_{L^{2}} \\
Ass \; 2: & \| u_{0} \|_{L^{1_{2}^{*}}} < \| W \|_{L^{1_{2}^{*}}}
\end{array}
\nonumber
\end{equation}
It was proved in \cite{kenmer} that if $Ass \; 0$ and $Ass \; 1$ hold then $u$ exists for all time and scatters as
$ t \rightarrow \pm \infty $. \\
Observe that $ Ass \; 0$ and $ Ass \; 1 $ hold if and only if $ Ass \; 0 $ and $ Ass \; 2$ hold. Indeed assume that $ Ass \; 0$ and $ Ass \; 1$ hold. Then the sharp Sobolev inequality (see Remark \ref{Rem:Imp}) and (\ref{Eqn:ValKW}) yield
$ K_{cr}(u_{0}) \geq \| \nabla u_{0} \|^{2}_{L^{2}} - \frac{1}{\| \nabla  W \|^{1_{2}^{*} -2}_{L^{2}}} \| \nabla u_{0} \|^{1_{2}^{*}}_{L^{2}} $
with $ K_{cr}(f)$ defined in (\ref{Eqn:DefK}). Hence
elementary considerations show that $ K_{cr}(u_{0}) \geq 0$.
Hence from  $ E_{cr}(u_{0}) =  \left( \frac{1}{2} - \frac{1}{1_{2}^{*}} \right) \| u_{0} \|^{1_{2}^{*}}_{L^{1_{2}^{*}}} + \frac{1}{2} K_{cr}(u_{0})$
and (\ref{Eqn:ValKW}) we see that $ Ass \; 2 $ holds. Conversely if $ Ass \; 0 $ and $ Ass \; 2 $ hold then it is clear that $ Ass \; 0 $ and $ Ass \; 1 $ hold.\\
Hence if $Ass \; 0$ and $Ass \; 2$ hold then $u$ exists for all time and scatters as
$ t \rightarrow \pm \infty$.

\subsection{Appendix $B$: scattering for small data}

Let $J \subset \mathbb{R}$ be an interval containing $0$ and such that $J \subsetneqq I_{max} $. Observe from
the embedding $ \| f \|_{L^{\frac{2(n+2)}{n-2}}} \lesssim \| D f \|_{L^{\frac{2n(n+2)}{n+4}}} $ and the
interpolation of $\| D u \|_{L_{t}^{\frac{2n(n+2)}{n^{2}+4}} L_{x}^{\frac{2n(n+2)}{n^{2}+4}} (J)} $ between
$ \| D u \|_{L_{t}^{\infty} L_{x}^{2}(J)} $ and  $ \| D u \|_{L_{t}^{\frac{2(n+2)}{n}} L_{x}^{\frac{2(n+2)}{n}} (J)} $ that
$\| u \|_{L_{t}^{\frac{2(n+2)}{n-2}} L_{x}^{\frac{2(n+2)}{n-2}} (J)} \lesssim Q(J,u)$. (\ref{Eqn:Strich}), (\ref{Eqn:NonlinEst}), and Proposition \ref{Prop:EstHighReg} allow to show that if $ \| u_0 \|_{\tilde{H}^{k}} \ll 1$, then the solution $u$ constructed by Proposition \ref{Prop:LocalWell} satisfies the conclusions of Theorem  \ref{thm:main}. More precisely there exists a positive constant $C$ such that

\begin{equation}
\begin{array}{ll}
Q(J,u) & \lesssim \| u_0 \|_{\tilde{H}^{k}} +  \| D( |u|^{\frac{4}{n-2}} u g(|u|) ) \|_{L_{t}^{\frac{2(n+2)}{n+4}}  L_{x}^{\frac{2(n+2)}{n+4}}(J)}
+ \| D^{k}( |u|^{\frac{4}{n-2}} u g(|u|) ) \|_{L_{t}^{\frac{2(n+2)}{n+4}}  L_{x}^{\frac{2(n+2)}{n+4}} (J)}  \\
& \lesssim \| u_0 \|_{\tilde{H}^{k}} +   \langle Q(J,u) \rangle^{C} \| u \|^{\frac{4}{n-2}-}_{L_{t}^{\frac{2(n+2)}{n-2}} L_{x}^{\frac{2(n+2)}{n-2}}(J)}  \\
& \lesssim \| u_0 \|_{\tilde{H}^{k}} +  \langle Q(J,u) \rangle^{C}  Q^{\frac{4}{n-2}-}(J,u) \cdot
\end{array}
\nonumber
\end{equation}
A continuity argument applied to the estimate above shows that $I_{max} = \mathbb{R}$ and that $Q (\mathbb{R},u) \lesssim \| u_{0} \|_{\tilde{H}^{k}} < \infty$. \\
Let $1 \gg \epsilon > 0 $. There exists $A(\epsilon)$ large enough such that if $t_2 \geq t_1 \geq A(\epsilon)$ then
$\| u \|_{L_{t}^{\frac{2(n+2)}{n-2}} L_{x}^{\frac{2(n+2)}{n-2}} ([t_1,t_2])} \ll \epsilon$. Hence
$\left\| e^{-it_{1} \triangle} u(t_1) - e^{-i t_{2} \triangle} u(t_2) \right\|_{\tilde{H}^{k}} \ll \epsilon$ by a similar
estimate to (\ref{Eqn:Diff}). Hence (\ref{Eqn:Scatt}) holds.




\subsection{Appendix $C$: comments on paper \cite{triroyrad}}
\label{Subsec:CommPropTriroyRad}

\subsubsection{Proposition $7$ in \cite{triroyrad} }

The proposition stated below is a slight modification of Proposition $7$ in \cite{triroyrad}:

 \begin{prop}
 Let $ 0  \leq \alpha < 1$, $k'$ and $\beta$ be integers such that $k' \geq 2$ and $\beta > k'-1 $, $(r ,r_{2}) \in (1,\infty)^{2}$,
 $(r_{1},r_{3}) \in (1, \infty]^{2}$ be such that $\frac{1}{r}= \frac{\beta}{r_{1}} + \frac{1}{r_{2}} +\frac{1}{r_{3}}$. Let $F: \mathbb{R}^{+} \rightarrow \mathbb{R}$
 be a $\mathcal{C}^{k'}-$ function, let $\tilde{F}: \mathbb{R}^{+} \rightarrow \mathbb{R}$ be a nondecreasing function, and let $ G := \mathbb{R}^{2} \rightarrow \mathbb{R}^{2} $ be a $\mathcal{C}^{k^{'}}- $ function such that

 \begin{equation}
 \begin{array}{l}
 (a): \;
 \left\{
 \begin{array}{l}
 \tau \in [0,1]: \; \left| \tilde{F} \left( |\tau x + (1-\tau)y|^{2} \right) \right|
 \lesssim  \left| \tilde{F}(|x|^{2}) \right| + \left| \tilde{F}( |y|^{2}) \right|, \,  \text{and} \\
 i \in \{0,...,k^{'} \}:  \; F^{[i]}(x) = O \left( \frac{\tilde{F}(x)}{x^{i}} \right)
 \end{array}
 \right.
 \\
 (b): \; i \in \{0,...,k^{'} \}: \;  G^{[i]}(x,\bar{x}) = O \left( |x|^{\beta + 1  - i} \right) \cdot
 \end{array}
 \nonumber
 \end{equation}
 Then

 \begin{equation}
 \begin{array}{ll}
 \left\| D^{ k' -1 + \alpha} \left( G(f,\bar{f}) F(|f|^{2}) \right) \right\|_{L^{r}} & \lesssim \| f \|^{\beta}_{L^{r_{1}}} \| D^{k' -1  + \alpha} f \|_{L^{r_{2}}}
 \| \tilde{F}(|f|^{2}) \|_{L^{r_{3}}}
 \end{array}
 \label{Eqn:EstToProveFrac2}
 \end{equation}
 Here $F^{[i]}$ and $G^{[i]}$ denotes the $i^{th}-$ derivative of $F$ and $G$ respectively.

 \label{Prop:NonlinFracSmooth2}
 \end{prop}

\begin{rem}

We mention the differences between the statement of the proposition above and that of Proposition $7$ in \cite{triroyrad}:

\begin{itemize}

\item the $\mathcal{C}^{k}-$ regularity assumption of $\tilde{F}$ present in the statement of Proposition $7$ in \cite{triroyrad} has been removed from that of the proposition above. Indeed, while reading the proof of Proposition $7$ in \cite{triroyrad} we have observed that this assumption is not necessary to prove
    (\ref{Eqn:EstToProveFrac2}).

\item one has added an extra assumption, namely the nondecreasing property of $\tilde{F}$. Indeed, recall that we prove Proposition $7$ in \cite{triroyrad}
first in the particular case where $\tilde{F} := F$ and then in the general case. While reading the proof of Proposition $7$ in \cite{triroyrad} we have observed
that this extra assumption \footnote{Observe that this extra assumption does not invalidate the use of Proposition $7$ in \cite{triroyrad} since Proposition $7$ is only applied in this paper to nondecreasing functions of the form $\tilde{F}(x) := \log^{c} \log \left( 10 + x \right) $.} is necessary to apply some Leibnitz-type rules  \footnote{It is well-known (see e.g \cite{taylor} and references therein) that if $ \left| H^{'} \left(\tau x + (1- \tau) y \right) \right| \lesssim \tilde{H}(x) + \tilde{H}(y) $ holds for
 some well-chosen $\tilde{H}$, then the fractional Leibnitz-type composition rule
 $ \left\| D^{\alpha'} H(f) \right\|_{L^{q}} \lesssim \left\| \tilde{H}(f) \right\|_{L^{q_{1}}}  \left\| D^{\alpha'} f \right\|_{L^{q_{2}}}$ holds for
 $(\alpha',q ,q_{1},q_{2})$ such that $ 1 > \alpha' > 0 $, $ (q,q_{2}) \in (1, \infty)^{2} $, $ q_{1} \in ( 1, \infty]$ , and $\frac{1}{q} = \frac{1}{q_{1}} + \frac{1}{q_{2}} $.
 In the proof of Proposition $7$ of \cite{triroyrad} we need to apply this rule to control norms of the type
 $ \left\| D^{\alpha'} \left( \partial_{z} G(f,\bar{f}) F(|f|^{2}) \right) \right\|_{L^{q}}$. The nondecreasing property of $\tilde{F}$ shows that $ \tau \in [0,1]: \; \left| \left( \partial_{z} G  \, F \left( | \cdot |^{2} \right) \right)^{'} \left( \tau x + (1- \tau)y \right)  \right| \lesssim \tilde{H}(x) + \tilde{H}(y) $, with $\tilde{H}(x) :=  |x|^{\beta - 1} \tilde{F} (|x|^{2})$.} and to prove (\ref{Eqn:EstToProveFrac2}).

 \end{itemize}

\end{rem}

Hence, taking into account the above remark, the proposition above was proved in \cite{triroyrad}. Hence Proposition \ref{Prop:NonlinFracSmooth} holds for
$\beta > k^{'}-1$. It remains to prove Proposition \ref{Prop:NonlinFracSmooth}  for
$\beta = k^{'} - 1 $.

\subsubsection{Proposition $7$ in \cite{triroyrad} and technical error in \cite{triroyrad} }

In this paragraph we point out a technical error in \cite{triroyrad}. We then explain how to fix the error. Recall that
in order to apply use Proposition $7$ in \cite{triroyrad} we need to assume that $ \beta > k' - 1 $. In \cite{triroyrad} we combine for
$ n \in \{ 3,4 \} $ this proposition  ( with $ \beta := \frac{4}{n-2} $, $G(f, \bar{f}):= |f|^{\frac{4}{n-2}} f$, and
$\tilde{F}(x) := F(x) := \log^{c} \left( \log(10 + x) \right)$) with (\ref{Eqn:Strich}) to control norms
at $\dot{H}^{k}-$ regularity of the nonlinearity on small intervals $J$ for $ \frac{n}{2} <  k < \frac{n+2}{n-2}$: for example we
use estimates in \cite{triroyrad} such as
$
\left\| D^{k} \left(|u|^{\frac{4}{n-2}} u F(|u|^{2}) \right) \right\|_{L_{t}^{\frac{2(n+2)}{n+4}} L_{x}^{\frac{2(n+2)}{n+4}} (J) }
\lesssim \| u \|^{\frac{4}{n-2}}_{L_{t}^{\frac{2(n+2)}{n-2}} L_{x}^{\frac{2(n+2)}{n-2}} (J)} \| F(|u|^{2}) \|_{L_{t}^{\infty} L_{x}^{\infty} (J) } \cdot
$
While writing this paper we have observed that for $n=3$ the constraint $\beta > k' - 1$ imposes that $k' \leq 4$. Hence in order to use (\ref{Eqn:EstToProveFrac2}) one
must assume that $ k < 4 - 1 + (1-) < 4 $, which does not cover the full
range $ \frac{n}{2}  < k < \frac{n+2}{n-2}$. We can fix this error by using Proposition \ref{Prop:NonlinFracSmooth} of this paper in Section $2$: indeed the less restrictive constraint $\beta \geq k' - 1  $ allows for $n \in \{ 3,4 \}$ to use (\ref{Eqn:EstToProveFrac}) for $ \frac{n}{2} < k < \frac{n+2}{n-2} $.

\subsection{Appendix $D$: Proof of Proposition \ref{Prop:NonlinFracSmooth} for $\beta = k'-1$}
\label{SubsecProofLimit}

\subsubsection{A fractional Leibnitz-type rule}

We write below a composition rule:

\begin{prop}

Let $ 0 < \alpha < 1 $ and let $H \in \mathcal{C} \left( \mathbb{R}^{2} \right) \cap \mathcal{C}^{1} \left( \mathbb{R}^{2} - \{ (0,0) \} \right) $ that satisfies the following condition: for all $(x,y,\tau) \in \mathbb{R}^{2} \times [0,1] $ such that $ \tau x + (1- \tau) y \neq 0 $, $ \left|  H^{'}  \left( \tau x + (1- \tau) y \right) \right| \lesssim \tilde{H}(x) + \tilde{H}(y)$  for some $\tilde{H}: \mathbb{R}^{2} \rightarrow \mathbb{R}$ that is continuous at the origin $O$ and such that
$ \tilde{H}(O) \lesssim \tilde{H}(z)$ for all $z \in \mathbb{R}^{2}$. Then

\begin{equation}
\begin{array}{ll}
\left\| D^{\alpha} H(f) \right\|_{L^{q}} & \lesssim \| \tilde{H}(f) \|_{L^{q_{1}}} \| D^{\alpha} f \|_{L^{q_{2}}}
\end{array}
\label{Eqn:Leibn1}
\end{equation}
if $(q,q_{2}) \in (1,\infty)^{2}$, $q_{1} \in ( 1 , \infty] $, and $ \frac{1}{q}= \frac{1}{q_{1}} + \frac{1}{q_{2}}$.

\label{Prop:CompRule}
\end{prop}

\begin{rem}
In (\ref{Eqn:Leibn1}) and in the proof of Proposition \ref{Prop:CompRule} we abuse notation: we write $f$, $f(x)$, and $f(y)$ for $(f,\bar{f})$,
$\left( f(x), \overline{f(x)} \right)$, and $\left( f(y), \overline{f(y)} \right)$ respectively.
\end{rem}

\begin{proof}

Assume that $ (*): \left| H \left( f(y) \right) - H(f(x)) \right| \lesssim \left( \tilde{H}(f(x)) + \tilde{H}(f(y)) \right) \left| f(y) - f(x) \right|$ holds for all
$(x,y) \in \mathbb{R}^{2} \times \mathbb{R}^{2} $. Then, rewriting the proof of Proposition $5.1$, Chapter $2$ of \cite{taylor} we see that (\ref{Eqn:Leibn1}) holds. So it remains to prove $(*)$. To this end we consider the line segment $[f(x), f(y)]$. If  $O \notin [f(x),f(y)]$ then this follows from the fundamental theorem of calculus. Assume now that $O \in [f(x),f(y)]$. We may assume WLOG that $f(x) \neq f(y)$. Assume that $f(x) \neq 0$ and that $f(y) \neq 0$. Let $ \min \left( |f(x)|, |f(y)| \right) > \epsilon > 0 $. Choose two points $ ( z_{1}, z_{2} ) \in [f(x),f(y)] $ such that $ |z_{1}| \leq \epsilon $ and $z_{2}$ symmetric of $z_{1}$ with respect to $O$ such that for $ Q \in  \{ H, \tilde{H} \} $, $ \left| Q(z_{1}) - Q(z_{2}) \right| \leq \epsilon $. Swapping $z_{1}$ with $z_{2}$ if necessary, we see
that $|f(x) - f(y)| = |f(x) - z_{2}| + |z_{2} - z_{2}| + |z_{1} - f(y)|$ and that

\begin{equation}
\begin{array}{ll}
\left| H \left( f(y) \right) - H ( z_{1} ) \right|  & \lesssim \left( \tilde{H}(f(y)) + \tilde{H}(z_{1}) \right) | f(y) - z_{1}|, \; \text{and} \\
\left| H \left( f(x) \right) - H ( z_{2} ) \right|  & \lesssim \left( \tilde{H}(f(x)) + \tilde{H}(z_{2}) \right) | f(x) - z_{2}| \cdot
\end{array}
\label{Eqn:IneqH}
\end{equation}
By the triangle inequality and by letting $\epsilon \rightarrow 0 $  we see that $(*)$ holds. Assume now that  $f(x)=O$ and that $f(y) \neq O$. Let
$ |f(y)| > \epsilon > 0 $ we choose one point $z_{1} \in [O,f(y)] $ such that $ |z_{1}| \leq \epsilon $. Hence the first inequality of
(\ref{Eqn:IneqH}) holds. By letting $\epsilon \rightarrow 0$ we see that $(*)$ holds. A similar proof shows that  $(*)$ also holds if
$f(x) \neq O $ and $f(y) = O$.

\end{proof}

We then recall a product rule (see \cite{taylor} and references therein).  Let $ \alpha' \geq 0$. Then

\begin{equation}
\begin{array}{ll}
\left\| D^{\alpha'} (fg) \right\|_{L^{q}} & \lesssim  \| D^{\alpha'} f \|_{L^{q_{1}}} \| g \|_{L^{q_{2}}} + \| f \|_{L^{q_{3}}}
\| D^{\alpha'} g \|_{L^{q_{4}}}
\end{array}
\nonumber
\end{equation}
if $ (q,q_{1},q_{4}) \in (1,\infty)^{3} $ and $(q_{2},q_{3})  \in (1,\infty]^{2}$ .

\subsubsection{Proof}

We slightly modify an argument in \cite{triroyrad}. \\
\\
Let $k^{'} = 2$. Then (see \cite{triroyrad}) $ \left\| D^{2-1 + \alpha} \left( G(f,\bar{f}) F(|f|^{2}) \right) \right\|_{L^{r}} \lesssim A_{1} + A_{2} + A_{3}
$ with $ A_{1} := \left\|  D^{\alpha} \left( \partial_{z} G(f,\bar{f}) \nabla f F(|f|^{2}) \right) \right\|_{L^{r}} $,
$ A_{2} :=  \left\|  D^{\alpha} \left( \partial_{\bar{z}} G(f,\bar{f})  \overline{\nabla f} F(|f|^{2} \right)) \right\|_{L^{r}} $, and
$A_{3} :=  \left\| D^{\alpha} \left( F^{'}(|f|^{2}) \Re \left( \bar{f} \nabla f \right) G(f,\bar{f})  \right) \right\|_{L^{r}} $. \\
Let $ H (x) : = \partial_{z} G(x,\bar{x})  F(|x|^{2}) $ and $ \tilde{H}(x) :=  \left| \tilde{F} \left( |x|^{2} \right) \right|$. Then $ \left| H^{'}(\tau x + (1- \tau) y ) \right| \lesssim \tilde{H}(x)  + \tilde{H}(y) $. The composition rule and the product rule show that

\begin{equation}
\begin{array}{ll}
A_{1} & \lesssim  \left\| D^{\alpha} \left( \partial_{z} G(f,\bar{f}) F(|f|^{2}) \right) \right\|_{L^{r_{4}}} \| D f \|_{L^{r_{5}}}
+ \left\|  \partial_{z} G(f,\bar{f}) F(|f|^{2})   \right\|_{L^{r_{6}}} \left\| D^{(2-1) +  \alpha}  f \right\|_{L^{r_{2}}} \\
& \lesssim \left\| \tilde{F}(|f|^{2}) \right\|_{L^{r_{3}}} \left\| D^{\alpha} f \right\|_{L^{r_{8}}} \left\| D f \right\|_{L^{r_{5}}}
+ \| f \|_{L^{r_{1}}} \left\| D^{(2-1) + \alpha} f \right\|_{L^{r_{2}}} \left\| \tilde{F}(|f|^{2}) \right\|_{L^{r_{3}}} \cdot
\end{array}
\nonumber
\end{equation}
Here $\frac{1}{r} = \frac{1}{r_{4}} + \frac{1}{r_{5}} $, $ \frac{1}{r} = \frac{1}{r_{6}} + \frac{1}{r_{2}}$, $\frac{1}{r_{4}} =\frac{1}{r_{3}} + \frac{1}{r_{8}} $, $\frac{1}{r_{5}}  = \frac{1 - \theta_{1}}{r_{1}} +  \frac{\theta_{1}}{r_{2}}$, and  $\theta_{1} = \frac{1}{1+ \alpha} $.
Notice that these relations imply that $\frac{1}{r_{8}} = \frac{\theta_{1}}{r_{1}} + \frac{1 - \theta_{1}}{r_{2}}$. Hence

\begin{equation}
\begin{array}{l}
\| D^{\alpha} f \|_{L^{r_{8}}} \lesssim \| f \|^{\theta_{1}}_{L^{r_{1}}} \| D^{(2-1) + \alpha} f \|^{1- \theta_{1}}_{L^{r_{2}}}, \; \text{and} \\
\| D f \|_{L^{r_{5}}} \lesssim \| f \|^{1 - \theta_{1}}_{L^{r_{1}}} \| D^{(2-1) + \alpha} f \|^{\theta_{1}}_{L^{r_{2}}} \cdot
\end{array}
\nonumber
\end{equation}
Hence $A_{1} \lesssim \text{R.H.S of (\ref{Eqn:EstToProveFrac})} $. Similarly
$A_{2} \lesssim \text{R.H.S of (\ref{Eqn:EstToProveFrac})} $. We also have

\begin{equation}
\begin{array}{ll}
A_{3} \lesssim  \sum \limits_{\tilde{f} \in \{ f ,\bar{f}\}} \left\| D^{\alpha} \left(  F^{'} (|f|^{2}) \tilde{f}  G(f, \bar{f}) \right) \right\|_{L^{r_{4}}}
\| D f \|_{L^{r_{5}}}  +  \sum \limits_{\tilde{f} \in \{ f ,\bar{f}\}}  \| D^{(2-1) + \alpha} f \|_{L^{r_{2}}} \| F^{'}(|f|^{2}) \tilde{f} G(f,\bar{f}) \|_{L^{r_{6}}} \\
\lesssim  A_{3,1} + A_{3,2}
\end{array}
\nonumber
\end{equation}
We have $A_{3,2} \lesssim  \text{R.H.S of (\ref{Eqn:EstToProveFrac})}$. The composition rule shows that $A_{3,1} \lesssim  \text{R.H.S of (\ref{Eqn:EstToProveFrac})}$. \\
\\
Assume that Proposition \ref{Prop:NonlinFracSmooth} holds for $k'$. Let us prove that Proposition \ref{Prop:NonlinFracSmooth} also holds for $k' + 1$. Following
\cite{triroyrad} we have  $ \left\| D^{k^{'} + \alpha} \left( G(f,\bar{f}) F(|f|^{2}) \right) \right\|_{L^{r}} \lesssim A^{'}_{1} + A^{'}_{2} + A^{'}_{3}
$ with $ A^{'}_{1} := \left\|  D^{k^{'} - 1 + \alpha} \left( \partial_{z} G(f,\bar{f}) \nabla f F(|f|^{2}) \right) \right\|_{L^{r}} $,
$ A^{'}_{2} :=  \left\|  D^{k^{'} -1 + \alpha} \left( \partial_{\bar{z}} G(f,\bar{f})  \overline{\nabla f} F(|f|^{2}) \right) \right\|_{L^{r}} $, and
$A^{'}_{3} :=  \left\| D^{k^{'} - 1 + \alpha} \left( \Re ( \bar{f} \nabla f ) G(f,\bar{f}) \right) \right\|_{L^{r}} $. We estimate $A^{'}_{3}$. Applying the induction assumption to the functions $\check{F}(x) := x F^{'}(x)$, $ G_{1} (x, \bar{x}) := \frac{G(x,\bar{x})}{x} $,  $ G_{2} (x, \bar{x}) := \frac{G(x,\bar{x})}{x}$ , and $F$ we see that

\begin{equation}
\begin{array}{ll}
A^{'}_{3} & \lesssim \sum \limits_{\tilde{f} \in \{f,\bar{f}\}}  \left\| D^{k^{'}-1 + \alpha} \left[ G(f,\bar{f}) F^{'}(|f|^{2}) \tilde{f} \right]  \right\|_{L^{r^{'}_{4}}} \| D f \|_{L^{r^{'}_{5}}} + \| D^{k^{'} + \alpha} f
\|_{L^{r_{2}}} \| G(f,\bar{f}) F^{'}(|f|^{2}) \tilde{f} \|_{L^{r_{6}}} \\
& \lesssim \| f \|^{\beta - 1}_{L^{r_{1}}} \| D^{k^{'}-1 + \alpha} f \|_{L^{r_{8}^{'}}}
\| \tilde{F}(|f|^{2}) \|_{L^{r_{3}}} \| D f \|_{L^{r_{5}^{'}}} + \| D^{k^{'} + \alpha} f
\|_{L^{r_{2}}} \| f
\|^{\beta}_{L^{r_{1}}} \| \tilde{F}(|f|^{2}) \|_{L^{r_{3}}} \\
& \lesssim \| f \|^{\beta}_{L^{r_{1}}} \| D^{k^{'} + \alpha} f \|_{L^{r_{2}}} \| \tilde{F}(|f|^{2}) \|_{L^{r_{3}}},
\end{array}
\nonumber
\end{equation}
where at the last line we used

\begin{equation}
\begin{array}{l}
\| D f \|_{L^{r_{5}^{'}}} \lesssim \| f \|^{1 - \theta_{1}^{'}}_{L^{r_{1}}} \| D^{(k^{'}+1)-1 + \alpha} f \|^{\theta_{1}^{'}}_{L^{r_{2}}}, \; \text{and} \\
\| D^{k' -1+ \alpha} f \|_{L^{r_{8}^{'}}} \lesssim \| f \|^{\theta_{1}^{'}}_{L^{r_{1}}} \| D^{(k^{'}+1)-1 + \alpha} f \|^{1 - \theta_{1}^{'}}_{L^{r_{2}}} \cdot
\end{array}
\nonumber
\end{equation}
Here $\frac{1}{r_{5}^{'}} = \frac{1- \theta_{1}^{'}}{r_{1}} + \frac{\theta_{1}^{'}}{r_{2}}$ with
$\theta_{1}^{'} = \frac{1}{k^{'} + \alpha}$, $\frac{1}{r} = \frac{1}{r_{4}^{'}} + \frac{1}{r_{5}^{'}}$, and $\frac{1}{r_{4}^{'}}
=  \frac{\beta -1 }{r_{1}} + \frac{1}{r_{8}^{'}} + \frac{1}{r_{3}} $. Observe that these relations imply that
$\frac{1}{r_{8}^{'}} = \frac{\theta_{1}^{'}}{r_{1}} + \frac{1- \theta_{1}^{'}}{r_{2}}$. We also have

\begin{equation}
\begin{array}{ll}
A_{1}^{'} & \lesssim \| D^{k' + \alpha } f \|_{L^{r_{2}}} \| \partial_{z} G(f,\bar{f}) F(|f|^{2}) \|_{L^{r_{6}}}
+ \left\| D^{k'-1 + \alpha} \left( \partial_{z} G(f,\bar{f}) F(|f|^{2}) \right) \right\|_{L^{r_{4}^{'}}} \| D f \|_{L^{r_{5}^{'}}} \\
& \lesssim  \| f \|^{\beta}_{L^{r_{1}}} \| D^{k^{'} + \alpha} f \|_{L^{r_{2}}} \| \tilde{F}(|f|^{2}) \|_{L^{r_{3}}} + \| f \|^{\beta -1}_{L^{r_{1}}}  \| D^{k'-1 + \alpha} f \|_{L^{r_{8}^{'}}}  \| \tilde{F}(|f|^{2}) \|_{L^{r_{3}}} \| D f \|_{L^{r_{5}^{'}}} \\
& \lesssim  \| f \|^{\beta}_{L^{r_{1}}} \| D^{k^{'} + \alpha} f \|_{L^{r_{2}}} \| \tilde{F}(|f|^{2}) \|_{L^{r_{3}}} \cdot
\end{array}
\nonumber
\end{equation}
Similarly $ A^{'}_{2} \lesssim \| f \|^{\beta}_{L^{r_{1}}} \| D^{k^{'} + \alpha} f \|_{L^{r_{2}}} \| \tilde{F}(|f|^{2}) \|_{L^{r_{3}}} $.

\end{document}